\RequirePackage{fix-cm}
\documentclass[smallextended]{svjour3}       

\smartqed
\usepackage{color}
\usepackage{eqnarray}
\usepackage{amssymb,amsmath}
\usepackage{multirow}
\usepackage[numbers,square]{natbib}
\newtheorem{thm}{Theorem}
\newtheorem{lem}[thm]{Lemma}

\newcommand{\llll}[1] {\left #1}
\newcommand{\rrrr}[1] {\right #1}

\newcommand{\dddd}[2]{\dfrac{#1}{#2}}

\newcommand{\aaaa}{\alpha}

\newcommand{\ssss}{\sigma}

\newcommand{\dddddd}{\delta}
\newcommand{\llllll}{\lambda}
\newcommand{\bbbb}{\beta}
\newcommand{\GGGG}{\Gamma}
\newcommand{\gggg}{\gamma}

\newcommand{\zzzz}{\zeta}

\newcommand{\be}{\begin{equation}}
\newcommand{\ee}{\end{equation}}
\newcommand{\baa}{\begin{align*}}
\newcommand{\eaa}{\end{align*}}
 \journalname{Computational and Applied Mathematics}
\begin{document}

\title{Asymptotic expansions and approximations for the Caputo derivative
\thanks{The third author is supported by the Bulgarian Academy of Sciences through the Program for Career Development of Young Scientists, Grant DFNP-17-88/28.07.2017, Project “Efficient Numerical Methods with an Improved Rate of Convergence for Applied Computational Problems”, by the Bulgarian National Fund of Science under Project DN 12/5-2017, Project “Efficient Stochastic Methods and Algorithms for Large-Scale Problems”, and  Project DN 12/4-2017, Project “Advanced Analytical and Numerical Methods for Nonlinear Differential Equations with Applications in Finance and Environmental Pollution”.
}
}
\author{Yuri Dimitrov \and Radan Miryanov \and Venelin Todorov}

\authorrunning{Y. Dimitrov \and R. Miryanov \and V. Todorov} 

\institute{
           Y. Dimitrov \at
					Department of Mathematics and Physics,
          University of Forestry, Sofia  1756, Bulgaria\\
              \email{yuri.dimitrov@ltu.bg}
           \and
					 R. Miryanov \at
			Department of Statistics and Applied Mathematics,
      University of Economics, Varna  9002, Bulgaria
              \email{miryanov@ue-varna.bg}
              \and
               V. Todorov \at
           Institute of Mathematics and Informatics, 
					Bulgarian Academy of Sciences,\\
					Acad. G. Bonchev Str, bl. 8, Sofia 1113, Bulgaria \email{vtodorov@math.bas.bg}\\
					Institute of Information and Communication   Technologies, 	Bulgarian Academy of Sciences,\\
					Acad. G. Bonchev Str, bl. 25A, Sofia 1113, Bulgaria \email{venelin@parallel.bas.bg}
 }

\date{Received: date / Accepted: date}

\maketitle

\begin{abstract}
In this paper we use the asymptotic expansions of the
binomial coefficients and the weights of the L1 approximation to obtain approximations of order $2-\alpha$
and second-order approximations of the Caputo derivative by modifying the weights of the shifted Gr\"{u}nwald-Letnikov difference approximation and the L1 approximation of the Caputo derivative. A modification of the shifted Gr\"{u}nwald-Letnikov approximation is obtained which allows second-order numerical solutions of fractional differential equations with arbitrary values of the solutions and their first derivatives at the initial point.
\end{abstract}

\section{Introduction}
 The  Caputo and Riemann-Liouville fractional derivatives are the two main approaches for generalizing the integer order derivatives. When $0<\aaaa<1$ the Caputo and Riemann-Liouville  derivatives with a lower limit at the point zero are defined as
$$y^{(\aaaa)}(x)=D^\aaaa y(x)=\dddd{1}{\GGGG(1-\aaaa)}\int_0^x \dddd{y'(t)}{(x-t)^\aaaa}dt,
$$
$$D_{RL}^\aaaa y(x)=\dddd{1}{\GGGG(1-\aaaa)}\dddd{d}{dx}\int_0^x \dddd{y(t)}{(x-t)^\aaaa}dt.
$$
 The Caputo and Riemann-Liouville derivatives are related as
$$\quad D_{RL}^\aaaa y(x)=D^\aaaa y(x)+\dddd{y(0)}{\GGGG(1-\aaaa)x^\aaaa}.
$$
 The Caputo derivative is a suitable choice for a fractional derivative in fractional differential equations. Fractional differential equations is a growing field of mathematics with applications in  finance, bioengineering, control theory, quantum mechanics (Magin 2004; Wang ang Xu 2007; Monje et al. 2010; Zhang et al. 2016).
The finite difference schemes for numerical solution of fractional differential equations involve approximations for the fractional derivative.  Let $h=x/N$ and $y_\bbbb=y(\bbbb h)$ for $0\leq\bbbb\leq N$. The Gr\"{u}nwald-Letnikov difference approximation is a first-order approximation of the Riemann-Liouville derivative, when $y\in C^1[0,x]$ and it is a first-order approximation of the Caputo derivative when the function $y$ satisfies the condition  $y(0)=0$:
\begin{equation*}
A^{GL}_N[y(x)]=\dddd{1}{h^\aaaa}\sum_{k=0}^{N-1} (-1)^k\binom{\aaaa}{k}y(x-k h)=y^{(\aaaa)}(x)+O( h).
\end{equation*}
When the function $y\in C^2[0,x]$ and satisfies the condition $y(0)=y'(0)=0$, the  Gr\"{u}nwald-Letnikov  approximation is a second-order approximation for the Caputo derivative at the point $x-\aaaa h/2$:
\begin{equation}\label{L1}
\dddd{1}{h^\aaaa}\sum_{k=0}^{N-1} (-1)^k\binom{\aaaa}{k}y(x-k h)=y^{(\aaaa)}(x-\aaaa h/2)+O\llll( h^2\rrrr).
\end{equation}
The weights $w_k^{(\aaaa)}=(-1)^k\binom{\aaaa}{k}$ of the Gr\"{u}nwald-Letnikov  approximation involve the binomial coefficients defined as:
$$\binom{\aaaa}{k}=\dddd{\GGGG(\aaaa+1)}{\GGGG(k+1)\GGGG(\aaaa-k+1)}=\dddd{\aaaa(\aaaa-1)\cdots(\aaaa-k+1)}{k!}.
$$
The binomial coefficients satisfy the identities
$$(-1)^k\binom{\aaaa}{k}=\binom{k-\aaaa-1}{k}=\dddd{\GGGG(k-\aaaa)}{\GGGG(k+1)\GGGG(-\aaaa)}, \binom{\aaaa-1}{k-1}+\binom{\aaaa-1}{k}=\binom{\aaaa}{k},
$$
and the  gamma function satisfies the asymptotic formula (Podlubny 1999)
$$k^{b-a}\dddd{\GGGG(k+a)}{\GGGG(k+b)}=1+O\llll(k^{-1} \rrrr).
$$
The weights of the Gr\"{u}nwald-Letnikov approximation are the coefficients of the binomial series of the function $(1-x)^\aaaa$ and the coefficients of the right endpoint asymptotic expansion of the  Gr\"{u}nwald-Letnikov approximation are equal to the coefficients of the series expansion of the function $(1-e^{-x})^\aaaa/x^\aaaa$ at the point $x=0$:
$$\llll(\dddd{1-e^{-x}}{x} \rrrr)^\aaaa=\sum_{k=0}^\infty \dddd{B_k^{(-\aaaa)}(-\aaaa)}{k!}x^k,
$$
where $B_k^{(-\aaaa)}(x)$ are the generalized Bernoulli polynomials.When the function $y\in C^m[0,x]$ and satisfies the condition $y^{(k)}(0)=0$, for $k=0,1,\dots,m$, the Gr\"{u}nwald-Letnikov approximation has an asymptotic expansion of order $m$:
\begin{equation*}
\dddd{1}{h^\aaaa}\sum_{k=0}^{N-1} (-1)^k\binom{\aaaa}{k}y(x-k h)=y^{(\aaaa)}(x)+\sum_{k=1}^{m-1}\dddd{B_k^{(-\aaaa)}(-\aaaa)}{k!}y^{(k+\aaaa)}(x)h^k+O\llll( h^m\rrrr),
\end{equation*}
where $y^{(k+\aaaa)}(x)$ is the Caputo derivative of order $k+\aaaa$ of the function $y$
$$y^{(k+\aaaa)}(x)=\dddd{1}{\GGGG(1-\aaaa)}\int_0^x \dddd{y^{(k+1)}(t)}{(x-t)^\aaaa}dt.
$$
 The asymptotic expansion formula of the Gr\"{u}nwald-Letnikov approximation is obtained from the series expansion of the Fourier transform of the approximation. Lubich (1986) constructs higher-order approximations of the fractional derivative which are derived from the Fourier transform of the approximation and the properties of the generating function.
Second-order and trird-order approximations of the Caputo and Riemann-Louville fractional derivatives related to the Gr\"unwald-Letnikov difference approximation and their applications for numerical solution of fractional differential equations are studied  in  (Tadjeran et al. 2006; Dimitrov 2014; Vong and Wang 2014; Tian et al. 2015; Gao et al. 2015;  Ren and Wang 2017). High-order approximations of the fractional derivative whose generating function is related to the generating function  of the Gr\"{u}nwald-Letnikov approximation are discussed in  (Chen and Deng 2014, Ding and Li 2016, 2017). Another approach for constructing approximations of the Caputo derivative uses Lagrange interpolation of the function and computation of the fractional integrals on the stencils of the grid. The L1 approximation is an important and commonly used approximation of the Caputo derivative  (Zhuang and  Liu 2006; Lin and Xu 2007; Jin et al. 2016).
\begin{equation}\label{L2}
\dddd{1}{h^\aaaa}\sum_{n=0}^N \ssss_n^{(\aaaa)}y(x-n h)=y^{(\aaaa)}(x)+O\llll( h^{2-\aaaa}\rrrr),
\end{equation}
where $\ssss_0^{(\aaaa)}=1/\GGGG(2-\aaaa),\quad \ssss_N^{(\aaaa)}=\llll((N-1)^{1-\aaaa}-N^{1-\aaaa}\rrrr)/\GGGG(2-\aaaa)$ and
$$\ssss_k^{(\aaaa)}=\dddd{(k-1)^{1-\aaaa}-2k^{1-\aaaa}+(k+1)^{1-\aaaa}}{\GGGG(2-\aaaa)},\qquad(k=1,2,\dots,N-1).
$$
When the function $y\in C^2[0,x]$, the L1 approximation of the Caputo derivative has an accuracy $O\llll( h^{2-\aaaa}\rrrr)$. In (Dimitrov 2016) we obtain the second-order asymptotic expansion formula of the L1 approximation
\begin{equation}\label{L6}
\dddd{1}{h^\aaaa}\sum_{k=0}^N \ssss_k^{(\aaaa)} y(x-kh)=y^{(\aaaa)}(x)+\dddd{\zzzz(\aaaa-1)}{\GGGG(2-\aaaa)}y^{\prime\prime}(x) h^{2-\aaaa}+O\llll(h^2\rrrr)
\end{equation}
 and a second-order  approximation of the Caputo derivative
\begin{equation}\label{L7}
\dddd{1}{h^\aaaa}\sum_{n=0}^N \dddddd_k^{(\aaaa)}y(x-k h)=y^{(\aaaa)}(x)+O\llll( h^2\rrrr),
\end{equation}
where $\dddddd_k^{(\aaaa)}=\ssss_k^{(\aaaa)}$ for $3\leq k\leq N$ and
$$\dddddd_0^{(\aaaa)}=\ssss_0^{(\aaaa)}-\dddd{\zzzz(\aaaa-1)}{\GGGG(2-\aaaa)},\;\dddddd_1^{(\aaaa)}=\ssss_1^{(\aaaa)}+\dddd{2\zzzz(\aaaa-1)}{\GGGG(2-\aaaa)},\; \dddddd_2^{(\aaaa)}=\ssss_2^{(\aaaa)}-\dddd{\zzzz(\aaaa-1)}{\GGGG(2-\aaaa)}.$$
When $k>2$, the weights $\dddddd_k^{(\aaaa)}$ of approximation \eqref{L7} are equal to the weights of the L1 approximation  and  the first three weights   are modified with the value of the zeta function at the point $\aaaa-1$. The asymptotic expansions of order $2+\aaaa$ of the weights of the Gr\"{u}nwald-Letnikov and the L1 approximations of the Caputo derivative are obtained from the binomial series expansion formula and the asymptotic formula for the gamma function:
$$w_k^{(\aaaa)}\sim \dddd{1}{\GGGG(-\aaaa)k^{1+\aaaa}},\quad \ssss_k^{(\aaaa)}\sim \dddd{1}{\GGGG(-\aaaa)k^{1+\aaaa}}.
$$
 The L1 approximation is constructed by approximating the first derivative by its value at the midpoint of an uniform grid. Gao et al. (2014) and Alikhanov (2015) construct L$1-2$ and L$2-1_\ssss$ approximations of the Caputo derivative which have accuracy $O\llll(h^{3-\aaaa}\rrrr)$.  Higher-order approximations of the Caputo derivative, related to the construction of the L1 approximation are studied in (Li et al. 2011; Yan et al. 2014; Zheng et al. 2017). In the last two decades many of the methods used for numerical solution of ordinary and partial differential equations have been applied for numerical solution of fractional differential equation; which include spline collocations methods (Pedas and Tamme 2014), Petrov-Galerkin spectral methods (Zayernouri and  Karniadakis 2014), orthogonal Legandre and  Laguerre polynomials (Bhrawy et al. 2015, Ezz-Eldien et al. 2017), Adomian decomposition methods (El-Borai et al. 2015). 
In the present paper we study the asymptotic properties of the weights of the Gr\"{u}nwald-Letnikov difference approximation and approximations of the Caputo derivative related to the Gr\"{u}nwald-Letnikov approximation. The approximations  discussed in the paper are applied for construction of finite-difference schemes for numerical solution of ordinary fractional differential equations. The paper is organized as follows.
In section 2 and section 3  we construct approximations of the Caputo derivative which are obtained from the shifted Gr\"{u}nwald-Letnikov approximation and the L1 approximation of the Caputo derivative by replacing the weights whose index is greater than $\left\lceil N/5\right\rceil$ with the first two terms of their asymptotic expansions formulas
In section 4 we  obtain the second-order shifted approximation of the Caputo derivative:
	\begin{equation}\label{L8}
		\dddd{1}{h^\aaaa}\sum_{k=0}^n \gggg_k^{(\aaaa)}y_{n-k}=y^{(\aaaa)}_{n-\aaaa/2}+O\llll( h^{2}\rrrr),
		\end{equation}
where $\gggg_k^{(\aaaa)}=(-1)^k \binom{\aaaa}{k}$ for $0\leq k\leq n-2$ and
$$\gggg_{n-1}^{(\aaaa)}=(-1)^{n-2}\binom{\aaaa-2}{n-2}\dddd{n-2\aaaa}{1+\aaaa-n}+\dddd{(n-\aaaa/2)^{1-\aaaa}}{\GGGG(2-\aaaa)},$$
$$\gggg_{n}^{(\aaaa)}=(-1)^{n-2}\binom{\aaaa-2}{n-2}-\dddd{(n-\aaaa/2)^{1-\aaaa}}{\GGGG(2-\aaaa)}.$$
Approximation \eqref{L8} is obtained from the Gr\"{u}nwald-Letnikov approximation by modifying the last two weights. While the Gr\"{u}nwald-Letnikov approximation is a second-order shifted approximation of the Caputo derivative for the functions $y\in C^2[0,x_n]$ which satisfy the condition $y(0)=y'(0)=0$, approximation \eqref{L8} is a second-order approximation of the Caputo derivative $y^{(\aaaa)}_{n-\aaaa/2}$ for all function $y\in C^2[0,x_n]$. In section 5 we derive the expansion formulas of order $2+\aaaa$ of the weights $\gggg_{n-1}^{(\aaaa)}$ and $\gggg_{n}^{(\aaaa)}$ of approximation \eqref{L8}.
 \section{Asymptotic expansion formula for binomial coefficients and shifted approximations for the Caputo derivative}
The asymptotic expansion formula for the ratio of gamma functions  is studied in  (Tricomi and Erd\'{e}lyi 1951). The ratio of gamma functions satisfies:
\begin{equation}\label{RatioGamma}
\dddd{\GGGG(k+t)}{\GGGG(k+s)}=k^{t-s}\sum_{m=0}^{\infty} \dddd{(-1)^m B_m^{(t-s+1)}(t)(s-t)^{(m)}}{m!}\dddd{1}{k^m}.
\end{equation}
The generalized Bernoulli polynomials $B_m^{(\aaaa)}(x)$ are defined as the coefficients of the series expansion of the function $e^{x t} t^\aaaa/(e^t-1)^\aaaa$
$$e^{x t}\llll(\dddd{t}{e^t-1}\rrrr)^\aaaa=
\sum_{m=0}^\infty B_m^{(\aaaa)}(x)\dddd{t^m}{m!}.
$$
From \eqref{RatioGamma} with $t:=-\aaaa,s:=1$ we obtain the asymptotic expansion formula for the weights of the Gr\"{u}nwald-Letnikov approximation  (Elezovi\'{c} 2005)
$$w_k^{(\aaaa)}=(-1)^k\binom{\aaaa}{k}=\dddd{1}{\GGGG(-\aaaa)}
\sum_{m=0}^\infty \dddd{(-1)^m B_m^{(-\aaaa)}(-\aaaa)(\aaaa+1)^{(m)}}{m! k^{m+\aaaa+1}},
$$
where $(\aaaa+1)^{(m)}$ is the rising factorial
$$(\aaaa+1)^{(m)}=(\aaaa+1)(\aaaa+2)\cdots(\aaaa+m).
$$
The generalized Bernoulli polynomial $B_0^{(-\aaaa)}(-\aaaa)=1$ and
$$B_1^{(-\aaaa)}(-\aaaa)=-\frac{\aaaa}{2},B_2^{(-\aaaa)}(-\aaaa)=\frac{1}{12} \aaaa (1+3 \aaaa),B_3^{(-\aaaa)}(-\aaaa)=-\frac{1}{8} \aaaa^2 (1+\aaaa).
$$
When $\aaaa=1$ the generalized Bernoulli polynomials are equal to the Bernoulli polynomials. The gamma function satisfies the identity $\GGGG(x+1)=x\GGGG(x)$ and
$(-1)^m(\aaaa+1)^{(m)}/\GGGG(-\aaaa)=1/\GGGG(-m-\aaaa)$.
The weights $w_k^{(\aaaa)}$ of the Gr\"{u}nwald-Letnikov approximation have an asymptotic expansion of order $M+\aaaa+2$
\be\label{L9}
w_k^{(\aaaa)}=(-1)^k\binom{\aaaa}{k}=\sum_{m=0}^M \dddd{B_m^{(-\aaaa)}(-\aaaa)}{m!\GGGG(-m-\aaaa)}\dddd{1}{k^{m+\aaaa+1}}+O\llll(\dddd{1}{k^{M+\aaaa+2}}\rrrr).
\ee
From \eqref{L9} with $M=2$ we obtain the  asymptotic expansion of order $4+\aaaa$ of the weights of the Gr\"{u}nwald-Letnikov approximation:
\begin{align} \label{L10}
w_k^{(\aaaa)}=\dddd{1}{\GGGG(-\aaaa)k^{1+\aaaa}}-\dddd{\aaaa}{2\GGGG(-1-\aaaa)k^{2+\aaaa}}+\dddd{\aaaa(3\aaaa+1)}{24\GGGG(-2-\aaaa)k^{3+\aaaa}}+O\llll(\dddd{1}{k^{4+\aaaa}}\rrrr).
\end{align}
Now we  construct a second-order shifted approximation of the Caputo derivative by replacing the weights  of the Gr\"{u}nwald-Letnikov  approximation which have an index $k>\left\lceil N/5\right\rceil$ with the first two terms of expansion formula \eqref{L10}.
	\begin{equation*}\label{AE_GL}
	\bar{w}_k^{(\aaaa)}=\left\{
	\begin{array}{l l}
(-1)^k \binom{\aaaa}{k},& 0\leq k\leq \left\lceil N/5\right\rceil\\
\dddd{1}{\GGGG(-\aaaa)k^{1+\aaaa}}-\dddd{\aaaa}{2\GGGG(-1-\aaaa)k^{2+\aaaa}},&  \left\lceil N/5\right\rceil+1\leq k\leq N-1\\
	\end{array}
		\right .
	\end{equation*}
	In Theorem 1 we show that the approximation  with weights $\bar{w}_k^{(\aaaa)}$ is a second-order shifted approximation of the Caputo derivative.
	\begin{thm} Let $y\in C^2[0,x]$ and $y(0)=y'(0)=0$. Then
\begin{equation}\label{L11}
			\dddd{1}{h^\aaaa}\sum_{k=0}^N \bar{w}_k^{(\aaaa)}y(x-k h)=y^{(\aaaa)}(x-\aaaa h/2)+O\llll( h^{2}\rrrr).
			\end{equation}
	\end{thm}
	\begin{proof} Let $M=\max_{t\in[0,x]}|y(t)|$ and $C>0$,  such that when $k>\left\lceil N/5\right\rceil$
	$$\llll|(-1)^k \binom{\aaaa}{k} -\llll( \dddd{1}{\GGGG(-\aaaa)k^{1+\aaaa}}-\dddd{\aaaa}{2\GGGG(-1-\aaaa)k^{2+\aaaa}}\rrrr) \rrrr|<\dddd{C}{k^{3+\aaaa}}.$$
	The difference $E_N[y(x)]$ between approximation \eqref{L11} and the Gr\"{u}nwald-Letnikov  approximation satisfies the estimate:
	\begin{align*}
	E_N[y&(x)]= \llll|\dddd{1}{h^\aaaa}\sum_{k=0}^{N-1} \bar{w}_k^{(\aaaa)}y(x-k h)-\dddd{1}{h^\aaaa}\sum_{k=0}^{N-1} (-1)^k \binom{\aaaa}{k}y(x-k h)
	\rrrr|\leq \\
	&\dddd{1}{h^\aaaa}\sum_{k=\left\lceil N/5\right\rceil+1}^{N-1} \llll|\bar{w}_k^{(\aaaa)}-(-1)^k \binom{\aaaa}{k}\rrrr||y(x-k h)|< \dddd{CM}{h^\aaaa}\sum_{k=\left\lceil  N/5\right\rceil+1}^N \dddd{1}{k^{3+\aaaa}}.
\end{align*}
The function $1/x^{3+\aaaa}$ is decreasing and
$$\sum_{k=\left\lceil N/5\right\rceil+1}^{\infty} \dddd{1}{k^{3+\aaaa}}<\int_{\left\lceil N/5\right\rceil}^{\infty}\dddd{1}{x^{3+\aaaa}}dx=\llll[-\dddd{1}{(2+\aaaa)x^{2+\aaaa}}  \rrrr]_{\left\lceil N/5\right\rceil
}^\infty<\dddd{1}{\left\lceil N/5\right\rceil^{2+\aaaa}} .
$$
Let $C_1=CM \llll(5/x\rrrr)^{2+\aaaa}$. Then
	\begin{align}\label{L12}
	 E_N[y(x)]\leq \dddd{CM}{h^\aaaa}\dddd{1}{( N/5)^{2+\aaaa}}=\dddd{CM}{h^\aaaa}\dddd{5^{2+\aaaa}}{N^{2+\aaaa}}=C_1 h^{2}.
\end{align}
From the triangle inequality
	\begin{align*}
	\Bigg|\dddd{1}{h^\aaaa}\sum_{k=0}^{N-1} \bar{w}_k^{(\aaaa)}y(x-\aaaa h/2)-y&^{(\aaaa)}(x-k h)\Bigg|<\\
&\Big|A^{GL}_N[y(x)]-y^{(\aaaa)}(x-\aaaa h/2)\Big|+E_N[y(x)].
\end{align*}
From \eqref{L1} and \eqref{L12}
$$\dddd{1}{h^\aaaa}\sum_{k=0}^{N-1} \bar{w}_k^{(\aaaa)}y(x-kh)=y^{(\aaaa)}(x-\aaaa h/2)+O\llll(h^2 \rrrr).
$$
\qed
	\end{proof}
	The result of Theorem 1 can be generalized to the approximations for the Caputo derivative which are obtained from any approximation by modifying the weights which have an index greater than $\left\lceil N/p\right\rceil$  with the first  terms of their expansion formulas, where $p$ is a positive number.  When the function $y\in C^2[0,x_n]$ and satisfies the condition $y(0)=y'(0)=0$  the Gr\"{u}nwald-Letnikov approximation has a second-order  expansion formula:
\be\label{L13}
\dddd{1}{h^\aaaa}\sum_{k=0}^{n-1}(-1)^k\binom{\aaaa}{k}y_{n-k}=y^{(\aaaa)}_n-\dddd{\aaaa}{2}y^{(1+\aaaa)}_n h+O\llll(h^2 \rrrr)=y^{(\aaaa)}_{n-\aaaa/2}+O\llll(h^2 \rrrr).
\ee
	In (Dimitrov 2018) we derive the  expansion formula for the right endpoint of the approximation for the Caputo derivative which has weights $k^{-1-\aaaa}/\GGGG(-\aaaa)$. When the function $y\in C^2[0,x]$ and satisfies the condition $y(0)=y'(0)=0$, the approximation has an asymptotic expansion of order $3-\aaaa$
\begin{align}\label{L14}
\dddd{1}{h^\aaaa}\sum_{k=1}^{n-1} \dddd{y_{n-k}}{\GGGG(-\aaaa)k^{1+\aaaa}}=y^{(\aaaa)}_n+&\dddd{\zzzz(\aaaa+1)}{\GGGG(-\aaaa)}y_nh^{-\aaaa}-\\
&\dddd{\zzzz(\aaaa)}{\GGGG(-\aaaa)}y'_n h^{1-\aaaa}+\dddd{\zzzz(\aaaa-1)}{2\GGGG(-\aaaa)}y''_n h^{2-\aaaa}+O\llll(h^{3-\aaaa}\rrrr).\nonumber
\end{align}
Expansion formula \eqref{L14} is obtained by applying formal integration by parts to the fractional integral in the definition of the Caputo derivative and a Fourier transform to the approximation. The expansion formula for the left endpoint of approximation \eqref{L14} is obtained from the Euler-Mclaurin formula for the function $y(t)/(x-t)^{1+\aaaa}$. By substituting $\aaaa:=\aaaa+1$ in \eqref{L14} we obtain
\begin{align*}
\dddd{1}{h^{1+\aaaa}}\sum_{k=1}^{n-1} \dddd{y_{n-k}}{\GGGG(-1-\aaaa)k^{2+\aaaa}}=y^{(1+\aaaa)}_n+&\dddd{\zzzz(\aaaa+2)}{\GGGG(-1-\aaaa)}y_n h^{-1-\aaaa}-\\
 \dddd{\zzzz(1+\aaaa)}{\GGGG(-1-\aaaa)}y'_n h^{-\aaaa}&+\dddd{\zzzz(\aaaa)}{2\GGGG(-1-\aaaa)}y''_n h^{1-\aaaa}+O\llll(h^{2-\aaaa}\rrrr),\nonumber
\end{align*}
\begin{align}\label{L15}
\dddd{1}{h^{\aaaa}}\sum_{k=1}^{n-1} \dddd{y_{n-k}}{\GGGG(-1-\aaaa)k^{2+\aaaa}}=h y^{(1+\aaaa)}_n+&\dddd{\zzzz(\aaaa+2)}{\GGGG(-1-\aaaa)}y_n h^{-\aaaa}-\\
\qquad \dddd{\zzzz(1+\aaaa)}{\GGGG(-1-\aaaa)}y'_n h^{1-\aaaa}&+\dddd{\zzzz(\aaaa)}{2\GGGG(-1-\aaaa)}y''_n h^{2-\aaaa}+O\llll(h^{3-\aaaa}\rrrr).\nonumber
\end{align}
By multiplying \eqref{L15} by $-\aaaa/2$ and adding to \eqref{L14} we obtain
\begin{align}\label{L16}
&\dddd{1}{h^{\aaaa}}\sum_{k=1}^{n-1}\llll( \dddd{1}{\GGGG(-\aaaa)k^{1+\aaaa}}-\dddd{\aaaa}{2\GGGG(-1-\aaaa)k^{2+\aaaa}}\rrrr)y_{n-k}=y^{(\aaaa)}_n-\dddd{\aaaa h}{2} y^{(1+\aaaa)}_n+\nonumber\\
& \llll(\dddd{\zzzz(\aaaa+1)}{\GGGG(-\aaaa)}-\dddd{\aaaa\zzzz(\aaaa+2)}{2\GGGG(-1-\aaaa)}\rrrr)y_n h^{-\aaaa}-
\llll(\dddd{\zzzz(\aaaa)}{\GGGG(-\aaaa)}-\dddd{\aaaa\zzzz(\aaaa+1)}{2\GGGG(-1-\aaaa)}\rrrr)y'_n h^{1-\aaaa}+\nonumber\\
&\quad \dddd{1}{2}\llll(\dddd{\zzzz(\aaaa-1)}{\GGGG(-\aaaa)}-\dddd{\aaaa\zzzz(\aaaa)}{2\GGGG(-1-\aaaa)}\rrrr)y''_n h^{2-\aaaa}+O\llll(h^{3-\aaaa}\rrrr).
\end{align}
By substituting $y'_n=(y_n-y_{n-1})/h+O\llll(h\rrrr)$ in \eqref{L16} we obtain
\begin{align}\label{L17}
\dddd{1}{h^{\aaaa}}\sum_{k=1}^{n-1}&\llll(\dddd{1}{\GGGG(-\aaaa)k^{1+\aaaa}}-\dddd{\aaaa}{2\GGGG(-1-\aaaa)k^{2+\aaaa}}\rrrr)y_{n-k}=y^{(\aaaa)}_n-\dddd{\aaaa h}{2} y^{(1+\aaaa)}_n-\nonumber\\
& \dddd{1}{h^\aaaa}\llll(\dddd{\zzzz(\aaaa)}{\GGGG(-\aaaa)}-\dddd{\aaaa\zzzz(\aaaa+1)}{2\GGGG(-1-\aaaa)}\rrrr)(y_n-y_{n-1})+\\
&\dddd{1}{h^\aaaa}\llll(\dddd{\zzzz(\aaaa+1)}{\GGGG(-\aaaa)}-\dddd{\aaaa\zzzz(\aaaa+2)}{2\GGGG(-1-\aaaa)}\rrrr)y_n+O\llll(h^{2-\aaaa}\rrrr).\nonumber
\end{align}
From \eqref{L17}  we obtain the shifted approximation for the Caputo derivative of order $2-\aaaa$:
\begin{equation}\label{L18}
			\dddd{1}{h^\aaaa}\sum_{k=0}^{n-1} \widetilde{w}_k^{(\aaaa)}y_{n-k}=y^{(\aaaa)}_n-\dddd{\aaaa}{2}y^{(1+\aaaa)}_n h+O\llll(h^{2 -\aaaa}\rrrr)=y^{(\aaaa)}_{n-\aaaa/2}+O\llll(h^{2 -\aaaa}\rrrr).
			\end{equation}
where
$$\widetilde{w}_0^{(\aaaa)}=\dddd{1}{\GGGG(-\aaaa)}\llll(\zzzz(\aaaa)+\dddd{1}{2}(\aaaa-1)(\aaaa+2)\zzzz(\aaaa+1)-\dddd{1}{2}\aaaa(\aaaa+1)\zzzz(\aaaa+2)\rrrr),$$
						$$\widetilde{w}_1^{(\aaaa)}=\dddd{1}{\GGGG(-\aaaa)}\llll(\dddd{1}{2}(\aaaa^2+\aaaa+2)-\zzzz(\aaaa)-\dddd{1}{2}\aaaa(\aaaa+1)\zzzz(\aaaa+1)\rrrr),$$
			$$\widetilde{w}_k^{(\aaaa)}=\dddd{1}{\GGGG(-\aaaa)k^{1+\aaaa}}-\dddd{\aaaa}{2\GGGG(-1-\aaaa)k^{2+\aaaa}},\quad (k=2,\dots,n-1).$$
		When the function $y\in C^2[0,x_n]$ and satisfies the condition $y(0)=y'(0)=0$ approximation \eqref{L18} an approximation for $y^{(\aaaa)}_{n-\aaaa/2}$  with an accuracy $O\llll(h^{2-\aaaa} \rrrr)$. The weights of approximation \eqref{L18} are equal to the first two terms in the asymptotic expansion formula of the weights of the Gr\"{u}nwald-Letnikov approximation  $(-1)^k\binom{\aaaa}{k}$, when $k\geq 2$. The first two terms of the second-order expansion of the Gr\"{u}nwald-Letnikov approximation and the expansion formula  of order $2-\aaaa$ of approximation \eqref{L18} are equal.
			
By substituting $y''_n=\dddd{1}{h^2}\llll(y_n-2y_{n-1}+y_{n-2}\rrrr)/h^2+O\llll(h\rrrr)$ and
			$$y'_n=\dddd{1}{h}\llll(\dddd{3}{2}y_n-2y_{n-1}+\dddd{1}{2}y_{n-2}\rrrr)+O\llll(h^2\rrrr),$$
	in \eqref{L16}		we obtain the shifted approximation for the Caputo derivative:
\begin{equation}\label{L19}
			\dddd{1}{h^\aaaa}\sum_{k=0}^{n-1} \widehat{w}_k^{(\aaaa)}y_{n-k}=y^{(\aaaa)}_{n-\aaaa/2}+O\llll(h^{2}\rrrr)=y^{(\aaaa)}_n-\dddd{\aaaa}{2}y^{(1+\aaaa)}_n h+O\llll(h^{3 -\aaaa}\rrrr),
			\end{equation}
			where
	\begin{align*}					
					\widehat{w}_0^{(\aaaa)}=	-\dddd{1}{4 \GGGG(-\aaaa)}(2\zzzz(-1+&\aaaa)+(\aaaa+3)(\aaaa-2)\zzzz(\aaaa)-\\
					&(3\aaaa^2+3\aaaa-4)\zzzz(\aaaa+1)+2\aaaa (\aaaa+1)\zzzz(2+\aaaa)),
\end{align*}
	\begin{align*}
						\widehat{w}_1^{(\aaaa)}=\dddd{1}{2\GGGG(-\aaaa)}(2+\aaaa+\aaaa^2+2\zzzz&(-1+\aaaa)+\\
						&\left(\aaaa+\aaaa^2-4\right) \zzzz(\aaaa)-2\aaaa (\aaaa+1)\zzzz(1+\aaaa)),
 \end{align*}
	\begin{align*}
						\widehat{w}_2^{(\aaaa)}=\dddd{1}{4\GGGG(-\aaaa)}\Big(\frac{4+\aaaa+\aaaa^2}{2^{1+\aaaa}}-2&\zzzz(-1+\aaaa)-\\
						&\left(\aaaa^2+\aaaa-2\right)\zzzz(\aaaa)+\aaaa (1+\aaaa)\zzzz(1+\aaaa)\Big),
 \end{align*}
			$$\widehat{w}_k^{(\aaaa)}=\dddd{1}{\GGGG(-\aaaa)k^{1+\aaaa}}-\dddd{\aaaa}{2\GGGG(-1-\aaaa)k^{2+\aaaa}},\quad (k=3,\dots,n-1).$$
			When the function $y\in C^2[0,x_n]$ and satisfies the condition $y(0)=y'(0)=0$ approximation \eqref{L19} is a second-order approximation for the Caputo derivative at the point $x_n-\aaaa h/2$. The weights of approximation \eqref{L19} are equal to the first two terms in the expansion formula of $(-1)^k\binom{\aaaa}{k}$ for $k\geq 3$. The first two terms of the second-order asymptotic expansion of the Gr\"{u}nwald-Letnikov approximation are equal to the first two terms of the expansion of order $3-\aaaa$ of approximation \eqref{L19}.
	\section{Expansion formula for the weights of the L1 approximation}
 From the binomial formula:
$$(k+1)^{1-\aaaa}=k^{1-\aaaa}\llll(1+\dddd{1}{k}\rrrr)^{1-\aaaa}=n^{1-\aaaa}\llll(\sum_{m=0}^5\binom{1-\aaaa}{m}\dddd{1}{k^m}+O\llll(\dddd{1}{k^6}\rrrr)\rrrr),
$$
\be\label{L20}
(k+1)^{1-\aaaa}=k^{1-\aaaa}+\sum_{m=1}^5\binom{1-\aaaa}{m}\dddd{1}{k^{m+\aaaa-1}}+O\llll(\dddd{1}{k^{5+\aaaa}}\rrrr).
\ee
Similarly
\be\label{L21}
(k-1)^{1-\aaaa}=k^{1-\aaaa}+\sum_{m=1}^5(-1)^m \binom{1-\aaaa}{m}\dddd{1}{k^{m+\aaaa-1}}+O\llll(\dddd{1}{k^{5+\aaaa}}\rrrr).
\ee
From \eqref{L20} and \eqref{L21}
$$(k-1)^{1-\aaaa}-2k^{1-\aaaa}+(k+1)^{1-\aaaa}=2\binom{1-\aaaa}{2}\dddd{1}{k^{1+\aaaa}}+2\binom{1-\aaaa}{4}\dddd{1}{k^{3+\aaaa}}+O\llll(\dddd{1}{k^{5+\aaaa}}\rrrr),
$$
$$\GGGG(2-\aaaa)\ssss_k^{(\aaaa)}=-\dddd{(1-\aaaa)\aaaa}{k^{1+\aaaa}}+\dddd{(1-\aaaa)(-\aaaa)(-\aaaa-1)(-\aaaa-2)}{12k^{3+\aaaa}}+O\llll(\dddd{1}{k^{5+\aaaa}}\rrrr).
$$
The weights $\ssss_k^{(\aaaa)}$ of the L1 approximation have an asymptotic expansion of order $5+\aaaa$, for $1\leq k\leq n-1$
\be\label{L22}
\ssss_k^{(\aaaa)}=\dddd{1}{\GGGG(-\aaaa)k^{1+\aaaa}}+\dddd{1}{12\GGGG(-2-\aaaa)k^{3+\aaaa}}+O\llll(\dddd{1}{k^{5+\aaaa}}\rrrr).
\ee
The last weight $\ssss_n^{(\aaaa)}$ of the L1 approximation \eqref{L2} satisfies
$$\GGGG(2-\aaaa)\ssss_n^{(\aaaa)}=(n-1)^{1-\aaaa}-n^{1-\aaaa}.$$
 From \eqref{L21}
$$\GGGG(2-\aaaa)\ssss_n^{(\aaaa)}=-\binom{1-\aaaa}{1}\dddd{1}{n^\aaaa}+\binom{1-\aaaa}{2}\dddd{1}{n^{1+\aaaa}}-\binom{1-\aaaa}{3}\dddd{1}{n^{2+\aaaa}}+O\llll(\dddd{1}{n^{3+\aaaa}}\rrrr),
$$
\be\label{L23}
\ssss_n^{(\aaaa)}=-\dddd{1}{\GGGG(1-\aaaa)n^{\aaaa}}+\dddd{1}{2\GGGG(-\aaaa)n^{1+\aaaa}}-\dddd{1}{6\GGGG(-1-\aaaa)n^{2+\aaaa}}+O\llll(\dddd{1}{n^{3+\aaaa}}\rrrr).
\ee
Similarly to the construction of approximation \eqref{L11} we use  asymptotic expansion formulas \eqref{L22} and \eqref{L23} to obtain approximations of the Caputo derivative by modifying the weights $\ssss_k^{(\aaaa)}$ of the L1 approximation  and the second order approximation \eqref{L7}  when $k>\left\lceil N/5 \right\rceil$.
\be\label{L24}
\dddd{1}{h^\aaaa}\sum_{k=0}^n \bar{\ssss}_k^{(\aaaa)}y_{n-k}=y^{(\aaaa)}_n+O\llll( h^{2-\aaaa}\rrrr),
\ee
where $\bar{\ssss}_0^{(\aaaa)}=1/\GGGG(2-\aaaa)$ and
	\begin{equation}\label{L25}
	\bar{\ssss}_k^{(\aaaa)}=\left\{
	\begin{array}{l l}
\dddd{(k-1)^{1-\aaaa}-2k^{1-\aaaa}+(k+1)^{1-\aaaa}}{\GGGG(2-\aaaa)},& 1\leq k\leq \left\lceil N/5\right\rceil,\\
\dddd{1}{\GGGG(-\aaaa)k^{1+\aaaa}}+\dddd{\aaaa}{12\GGGG(-2-\aaaa)k^{3+\aaaa}},&  \left\lceil N/5\right\rceil< k\leq N-1.\\
	\end{array}
		\right .
	\end{equation}
	When $ n\leq \left\lceil N/5\right\rceil $ the last weight $\bar{\ssss}_n^{(\aaaa)}$ of approximation \eqref{L24} is equal to the last weight $\ssss_n^{(\aaaa)}$ of the L1 approximation. When $n>\left\lceil N/5\right\rceil$ the value of the weight $\bar{\ssss}_n^{(\aaaa)}$ is equal to the terms of expansion formula \eqref{L23} of  $\ssss_n^{(\aaaa)}$.
		\begin{equation}\label{L26}
	\bar{\ssss}_n^{(\aaaa)}=\left\{
	\begin{array}{l l}
\ssss_n^{(\aaaa)}=\dddd{(n-1)^{1-\aaaa}-n^{1-\aaaa}}{\GGGG(2-\aaaa)},& n\leq \left\lceil N/5\right\rceil,\\
-\dddd{n^{-\aaaa}}{\GGGG(1-\aaaa)}+\dddd{n^{-\aaaa-1}}{2\GGGG(-\aaaa)}-\dddd{n^{-\aaaa-2}}{6\GGGG(-1-\aaaa)},&  n>\left\lceil N/5\right\rceil.\\
	\end{array}
		\right .
	\end{equation}
From approximation \eqref{L7} we obtain the second-order approximation of the Caputo derivative
\be\label{L27}
\dddd{1}{h^\aaaa}\sum_{k=0}^n \bar{\dddddd}_k^{(\aaaa)}y_{n-k}=y^{(\aaaa)}_n+O\llll( h^{2}\rrrr),
\ee
where 	the  weights $\bar{\dddddd}_k^{(\aaaa)}=\bar{\ssss}_k^{(\aaaa)}$ for $k\geq 3$ are defined with \eqref{L25}, \eqref{L26}  and $\bar{\dddddd}_k^{(\aaaa)}=\dddddd_k^{(\aaaa)}$ for $k=0,1,2$: $\bar{\dddddd}_0^{(\aaaa)}=(1-\zzzz(\aaaa-1))/\GGGG(2-\aaaa)$,
$$ \bar{\dddddd}_1^{(\aaaa)}=\dddd{2-2^{1-\aaaa}+2\zzzz(\aaaa-1)}{\GGGG(2-\aaaa)},\; \bar{\dddddd}_2^{(\aaaa)}=\dddd{3^{1-\aaaa}-2^{-\aaaa}+1-\zzzz(\aaaa-1)}{\GGGG(2-\aaaa)}.$$
When the function $y\in C^2[0,x_n]$, approximations \eqref{L24}and \eqref{L27} have accuracy $O\llll( h^{2-\aaaa}\rrrr)$ and $O\llll( h^{2}\rrrr)$. The proof is similar to the proof of Theorem 1.

The fractional integral of order $\aaaa>0$ is defined as
$$I^\aaaa y(x)=y^{(-\aaaa)}(x)=\dddd{}{}\int_0^x (x-t)^{\aaaa-1}y(t)dt.$$

The fractional integral $I^\aaaa y(x)$ is the fractional derivative of the function $y(x)$ of order $-\aaaa$. The Riemann sum approximation of the fractional integral has an expansion of order $1+\aaaa$
$$h^\aaaa \sum_{k=1}^{n-1}\dddd{y_{n-k}}{\GGGG(\aaaa)k^{1-\aaaa}}= y_n^{(-\aaaa)}+\dddd{\zzzz(1-\aaaa)}{\GGGG(\aaaa)}y_n h^\aaaa+O\llll(h^{1+\aaaa} \rrrr),
$$
when $y(0)=0$. From the approximations discussed in the paper and other approximations for the fractional derivatives and integrals we observe that the weights of the approximation
 $h^{-\aaaa}\sum_{k=0}^{n}w_k^{(\aaaa)}y(x-kh)$ of the fractional derivative $y^{(\aaaa)}(x)$ satisfy $w_k^{(\aaaa)}\sim \dddd{c_k}{k^{1+\aaaa}}$ and weights of the approximations of the definite integral satisfy $w_k^{(-1)}\sim c_k$. This property holds for the trapezoidal approximation, Simpson's  approximation, and the quadrature formulas discussed in (Dimitrov, Miryanov and Todorov 2017).
	\section{Shifted Gr\"{u}nwald-Letnikov difference approximation}
In this section we use the method from (Dimitrov 2018) to derive an approximation \eqref{L8} for the Caputo derivative, which is obtained from the Gr\"{u}nwald-Letnikov approximation  by modifying the last two weights.  Approximation \eqref{L8} is a second-order shifted approximation for the Caputo derivative  $y^{(\aaaa)}_{n-\aaaa/2}$ for all functions $y\in C^2[0,x_n]$. Let $y_0=y(0),y'_0=y'(0)$ and
$$y(x)=y_0+y'_0 x+z(x).$$
  	The function  $z(x)$ satisfies the condition $z(0)=z'(0)=0$.
		$$A^{GL}_N [y(x)]=y_0 A^{GL}_N [1]+y'_0 A^{GL}_N [x]+ A^{GL}_N [z(x)],$$
		$$A^{GL}_N [y(x)]=y_0 A^{GL}_N [1]+y'_0 A^{GL}_N [x]+z^{(\aaaa)}\llll(x-\aaaa h/2 \rrrr)+O\llll( h^2\rrrr).$$
	The Caputo derivative of the functions $y$ and $z$ satisfies
	$$y^{(\aaaa)}\llll(x-\aaaa h/2 \rrrr)=\dddd{y'_0}{\GGGG(2-\aaaa)}(x-\aaaa h/2)^{1-\aaaa}+z^{(\aaaa)}\llll(x-\aaaa h/2 \rrrr).$$
Then
$$A^{GL}_N [y(x)]=y_0 A^{GL}_h [1]+y'_0 A^{GL}_h [x]-\dddd{y'_0 (x-\aaaa h/2)^{1-\aaaa}}{\GGGG(2-\aaaa)}+y^{(\aaaa)}\llll(x-\aaaa h/2 \rrrr)+O\llll( h^2\rrrr),$$
\begin{align}\label{L29}
A^{GL}_N [y(x)]=y_0 A^{GL}_N [1]+&\llll( A^{GL}_N [x]-\dddd{ (x-\aaaa h/2)^{1-\aaaa}}{\GGGG(2-\aaaa)}\rrrr)y'_0 +\\
&\qquad\qquad\qquad\qquad y^{(\aaaa)}\llll(x-\aaaa h/2 \rrrr)+O\llll( h^2\rrrr).\nonumber
\end{align}
	Denote $W_N^0=h^\aaaa A^{GL}_N [1]=\sum_{k=0}^{N-1}w_k^{(\aaaa)}$ and
	$$W_N^1=\dddd{1}{h^{1-\aaaa}}\llll(A^{GL}_h [x]-\dddd{ (x-\aaaa h/2)^{1-\aaaa}}{\GGGG(2-\aaaa)}\rrrr),$$
	$$W_N^1=\dddd{1}{h^{1-\aaaa}}\llll(\dddd{1}{h^\aaaa}\sum_{k=0}^{N-1}w_k^{(\aaaa)}(x-k h)-\dddd{h^{1-\aaaa}}{\GGGG(2-\aaaa)}(N-\aaaa/2)^{1-\aaaa}\rrrr),$$
\be\label{L30}
				W_N^1=N W_N^0-\sum_{k=0}^{N-1}k w_k^{(\aaaa)}-\dddd{(N-\aaaa/2)^{1-\aaaa}}{\GGGG(2-\aaaa)}.
\ee
In Lemma 2 we show that $W_N^1=O\llll(h^{1+\aaaa} \rrrr)$. From \eqref{L29}, the Gr\"{u}nwald-Letnikov approximation satisfies:
				$$A^{GL}_N [y(x)]=y^{(\aaaa)}\llll(x-\aaaa h/2 \rrrr)+\dddd{1}{h^\aaaa}\llll( W_N^0 y_0+W_N^1 y'_0 h\rrrr)+O\llll( h^2\rrrr).$$
	By substituting $ h y'_0=y_1-y_0+O\llll(h^2\rrrr)$ we obtain
$$A^{GL}_N [y(x)]=y^{(\aaaa)}\llll(x-\aaaa h/2 \rrrr)+\dddd{1}{h^\aaaa}\llll( W_N^0 y_0+W_N^1 (y_1-y_0)\rrrr)+O\llll( h^2\rrrr).$$
Denote by $\bar{A}^{GL}_N [y(x)]$ the approximation of the Caputo derivative
	$$\bar{A}^{GL}_N [y(x)]=A^{GL}_N [y(x)]-\dddd{1}{h^\aaaa}\llll( W_N^0 y_0+W_N^1 (y_1-y_0)\rrrr).$$	
	Approximation $\bar{A}^{GL}_N [y(x)]$  	is obtained from the Gr\"{u}nwald-Letnikov approximation by modifying the last two weights.
	$$\bar{A}^{GL}_N [y(x)]=\dddd{1}{h^\aaaa}\sum_{k=0}^{N}\gggg_k^{(\aaaa)}y(x-kh)=y^{(\aaaa)}(x-\aaaa h/2)+O\llll( h^2\rrrr),$$
where $\gggg_k^{(\aaaa)}=w_k^{(\aaaa)}=(-1)^k \binom{\aaaa}{k}$ for $k=0,1,\dots,N-2$ and
	$$\gggg_{N-1}^{(\aaaa)}=w_{N-1}^{(\aaaa)}-W_N^1=(-1)^{N-1}\binom{\aaaa}{N-1}-\sum_{k=0}^{N-1}(N-k) w_k^{(\aaaa)}+\dddd{(N-\aaaa/2)^{1-\aaaa}}{\GGGG(2-\aaaa)},$$
		$$\gggg_{N-1}^{(\aaaa)}=-\sum_{k=0}^{N-2} (-1)^k(N-k) \binom{\aaaa}{k}+\dddd{(N-\aaaa/2)^{1-\aaaa}}{\GGGG(2-\aaaa)},$$
		$$\gggg_{N}^{(\aaaa)}=W_N^1-W_N^0=\sum_{k=0}^{N-2}(N-k-1) w_k^{(\aaaa)}-\dddd{(N-\aaaa/2)^{1-\aaaa}}{\GGGG(2-\aaaa)}-w_{N-1}^{(\aaaa)}.$$
	By induction we can show that the  weights of the Gr\"{u}nwald-Letnikov approximation satisfy the identities (Podlubny 1999)
\be\label{L31}
W_N^0=\sum_{k=0}^{N-1} w_{k}^{(\aaaa)}=\sum_{k=0}^{N-1} (-1)^k \binom{\aaaa}{k}=(-1)^{N-1} \binom{\aaaa-1}{N-1}=w_{N-1}^{(\aaaa-1)},
\ee			
\be\label{L32}
\sum_{k=1}^{N-1} (-1)^k k\binom{\aaaa}{k}=-\aaaa \sum_{k=1}^{N-1} (-1)^{k-1} \binom{\aaaa-1}{k-1}=-\aaaa w_{N-2}^{(\aaaa-2)}.
\ee
		From \eqref{L30}, \eqref{L31} and \eqref{L32}
		$$W_N^1=N w_{N-1}^{(\aaaa-1)}+\aaaa w_{N-2}^{(\aaaa-2)}-\-\dddd{(N-\aaaa/2)^{1-\aaaa}}{\GGGG(2-\aaaa)},$$
\begin{align*}
N w_{N-1}^{(\aaaa-1)}+\aaaa w_{N-2}^{(\aaaa-2)}&=N(-1)^{N-1}\binom{\aaaa-1}{N-1}+\aaaa(-1)^{N-2}\binom{\aaaa-2}{N-2}=\\
&(-1)^{N-2}\llll(-\dddd{N(\aaaa-1)}{N-1}\binom{\aaaa-2}{N-2} +\aaaa\binom{\aaaa-2}{N-2}   \rrrr)=\\
&(-1)^{N-2}\binom{\aaaa-2}{N-2}\dddd{N-\aaaa}{N-1}=w_{N-1}^{(\aaaa-2)}.
\end{align*}
Hence
		$$W_N^1= w_{N-1}^{(\aaaa-2)}-\-\dddd{(N-\aaaa/2)^{1-\aaaa}}{\GGGG(2-\aaaa)}.$$
		The weight $\gggg_{N}^{(\aaaa)}$ of  satisfies
	$$\gggg_N^{(\aaaa)}=W_N^1-W_N^0= w_{N-1}^{(\aaaa-2)}- w_{N-1}^{(\aaaa-1)}-\dddd{(N-\aaaa/2)^{1-\aaaa}}{\GGGG(2-\aaaa)}.
	$$
From the properties of the binomial coefficients
\begin{align*}
w_{N-1}^{(\aaaa-2)}- w_{N-1}^{(\aaaa-1)}=
 (-1)^{N-1} \llll( \binom{\aaaa-2}{N-1}-\binom{\aaaa-1}{N-1}\rrrr)=(-1)^{N-2} \binom{\aaaa-2}{N-2}.
\end{align*}		
Hence
	$$\gggg_N^{(\aaaa)}=w_{N-2}^{(\aaaa-2)}-\dddd{(N-\aaaa/2)^{1-\aaaa}}{\GGGG(2-\aaaa)}.
	$$
	The weight $\gggg_{N-1}^{(\aaaa)}$  satisfies
		$$\gggg_{N-1}^{(\aaaa)}=w_{N-1}^{(\aaaa)}-W_N^1=w_{N-1}^{(\aaaa)}- w_{N-1}^{(\aaaa-2)}+\dddd{(N-\aaaa/2)^{1-\aaaa}}{\GGGG(2-\aaaa)},
	$$
	\begin{align*}
w_{N-1}^{(\aaaa)}- &w_{N-1}^{(\aaaa-2)}=(-1)^{N-1}\llll( \binom{\aaaa}{N-1}- \binom{\aaaa-2}{N-1}\rrrr)=\\
&(-1)^{N-1}\llll( \binom{\aaaa-2}{N-2}\dddd{\aaaa(\aaaa-1)}{(N-1)(\aaaa-N+1)}- \binom{\aaaa-2}{N-1}\dddd{\aaaa-N}{N-1}\rrrr)=\\
&(-1)^{N-2}\binom{\aaaa-2}{N-2}\dddd{N-2\aaaa}{1+\aaaa-N}.
\end{align*}
Hence
		$$\gggg_{N-1}^{(\aaaa)}=\dddd{N-2\aaaa}{1+\aaaa-N}w_{N-2}^{(\aaaa-2)}+\dddd{(N-\aaaa/2)^{1-\aaaa}}{\GGGG(2-\aaaa)}.
	$$	
 Approximation  $\bar{A}^{GL}_n [y(x)]$ is a second-order shifted approximation for the Caputo derivative, when the function $y\in C^2[0,x_n]$:
	\begin{equation}\label{L33}
		\dddd{1}{h^\aaaa}\sum_{k=0}^n \gggg_k^{(\aaaa)}y_{n-k}=y^{(\aaaa)}_{n-\aaaa/2}+O\llll( h^{2}\rrrr),
		\end{equation}
where  $\gggg_k^{(\aaaa)}=w_{k}^{(\aaaa)}=(-1)^k \binom{\aaaa}{k}$ for $0\leq k\leq n-2$  and
		$$\gggg_{n-1}^{(\aaaa)}=\dddd{n-2\aaaa}{1+\aaaa-n}w_{n-2}^{(\aaaa-2)}+\dddd{(n-\aaaa/2)^{1-\aaaa}}{\GGGG(2-\aaaa)}, \gggg_n^{(\aaaa)}=w_{n-2}^{(\aaaa-2)}-\dddd{(n-\aaaa/2)^{1-\aaaa}}{\GGGG(2-\aaaa)}.
	$$	
		\section{Asymptotic expansion formulas }
In this section we obtain the asymptotic expansions of $W_n^0= w_{n-1}^{(\aaaa-1)},W_n^1$ and the last two weights $\gggg_{n-1}^{(\aaaa)}$ and $\gggg_{n}^{(\aaaa)}$ of approximation \eqref{L33}, and we construct an approximation of the Caputo derivative by modifying the weights of approximation \eqref{L33} which have an index greater than $\left\lceil N/5 \right\rceil$  with the first terms of their asymptotic expansions.
\begin{lem}
\begin{align*}W_n^0=\dddd{1}{\GGGG(1- \aaaa) n^\aaaa}-\dddd{\aaaa+1}{2\GGGG(-\aaaa)n^{1+\aaaa}}+\dddd{(2+\aaaa)(1+3\aaaa)}{24\GGGG(-1-\aaaa)n^{2+\aaaa}}+O\llll(\dddd{1}{n^{3+\aaaa}}\rrrr).
\end{align*}
$$W_n^1=\dddd{\aaaa-2}{24\GGGG(-\aaaa)n^{1+\aaaa}}+O\llll(\dddd{1}{n^{2+\aaaa}}\rrrr).$$
\end{lem}
\begin{proof}
 From expansion formula \eqref{L10} for $w_{n}^{(\aaaa-1)}$ with $\aaaa:=\aaaa-1$
\begin{align}\label{L300}
w_{n-1}^{(\aaaa-1)}=\dddd{1}{\GGGG(1-\aaaa) (n-1)^\aaaa}-&\dddd{\aaaa-1}{2\GGGG(-\aaaa)(n-1)^{1+\aaaa}}+\\
&\qquad\dddd{(\aaaa-1)(3\aaaa-2)}{24\GGGG(-1-\aaaa)n^{2+\aaaa}}+O\llll(\dddd{1}{n^{3+\aaaa}}\rrrr).\nonumber
\end{align}
From the binomial formula
\be\label{L35}
\dddd{1}{(n-1)^{\aaaa}}=n^{-\aaaa}\llll(1-\dddd{1}{n}  \rrrr)^{-\aaaa}=\dddd{1}{n^{\aaaa}}+\dddd{\aaaa}{n^{1+\aaaa}}+\dddd{\aaaa(\aaaa+1)\aaaa}{2n^{2+\aaaa}} +O\llll(\dddd{1}{n^{3+\aaaa}}\rrrr).
\ee
By substituting $\aaaa:=\aaaa+1$ and $\aaaa:=\aaaa+2$ in \eqref{L35} we obtain
\begin{align}\label{L36}
\dddd{1}{(n-1)^{1+\aaaa}}=\dddd{1}{n^{1+\aaaa}}+\dddd{1+\aaaa}{n^{2+\aaaa}} +O\llll(\dddd{1}{n^{3+\aaaa}}\rrrr), \dddd{1}{(n-1)^{2+\aaaa}}=\dddd{1}{n^{2+\aaaa}} +O\llll(\dddd{1}{n^{3+\aaaa}}\rrrr).
\end{align}
From \eqref{L300}, \eqref{L35} and \eqref{L36} the weight
 $w_{n-1}^{(\aaaa-1)}=W_n^0$ has an asymptotic expansion
\begin{align*}w_{n-1}^{(\aaaa-1)}=\dddd{1}{\GGGG(1- \aaaa) n^{\aaaa}}+\llll(\dddd{\aaaa}{\GGGG(1-\aaaa)}-\dddd{\aaaa-1}{2\GGGG(-\aaaa)}\rrrr)\dddd{1}{n^{1+\aaaa}}+\dddd{S}{n^{2+\aaaa}}+O\llll(\dddd{1}{n^{3+\aaaa}}\rrrr),
\end{align*}
where
$$S=-\dddd{\aaaa(\aaaa+1)}{2\GGGG(1-\aaaa)}-\dddd{(\aaaa-1) (\aaaa+1)}{2\GGGG(-\aaaa)}-\dddd{(\aaaa+1)(2+\aaaa)}{24\GGGG(-1-\aaaa)}=\dddd{(\aaaa+2)(3\aaaa+1)}{24\GGGG(-1-\aaaa)}.$$
Hence
\begin{align}\label{L37}
w_{n-1}^{(\aaaa-1)}=\dddd{1}{\GGGG(1-\aaaa) n^{\aaaa}}-\dddd{1+\aaaa}{2\GGGG(-\aaaa)n^{1+\aaaa}}+\dddd{(\aaaa+2)(3\aaaa+1)}{24\GGGG (-1-\aaaa)n^{2+\aaaa}}+O\llll(\dddd{1}{n^{3+\aaaa}}\rrrr).
\end{align}
 From the binomial formula
$$\llll(n-\dddd{\aaaa}{2}\rrrr)^{1-\aaaa}=n^{1-\aaaa}\llll(1-\dddd{\aaaa}{2n}\rrrr)^{1-\aaaa}=n^{1-\aaaa}\llll(\sum_{k=0}^{2}\binom{1-\aaaa}{k}\llll(\dddd{\aaaa}{2n}\rrrr)^k+O\llll(\dddd{1}{n^3} \rrrr)\rrrr),$$
$$
	\dddd{\llll(n-\aaaa/2\rrrr)^{1-\aaaa}}{\GGGG(2-\aaaa)}=\dddd{1}{\GGGG(2-\aaaa)n^{\aaaa-1}}+\dddd{1}{2\GGGG(- \aaaa)n^{\aaaa}}+\dddd{\aaaa^2}{8\GGGG(-\aaaa)n^{\aaaa+1}}+O\llll(\dddd{1}{n^{2+\aaaa}} \rrrr).
$$
	 From \eqref{L37} with $\aaaa:=\aaaa-1$
		\begin{align*} \label{L39}
w_{n-1}^{(\aaaa-2)}=\dddd{1}{\GGGG(2-\aaaa)n^{\aaaa-1}}- \dddd{\aaaa}{2\GGGG(1-\aaaa)n^{\aaaa}}+\dddd{(\aaaa+1)(3\aaaa-2)}{24\GGGG(-\aaaa)n^{1+\aaaa}}+O\llll(\dddd{1}{n^{2+\aaaa}}\rrrr).\nonumber
\end{align*}
Hence
		$$W_n^1=w_{n-1}^{(\aaaa-2)}-\dddd{(n-\aaaa/2)^{1-\aaaa}}{\GGGG(2-\aaaa)}=\dddd{\aaaa-2}{24\GGGG(-\aaaa)n^{1+\aaaa}}+O\llll(\dddd{1}{n^{2+\aaaa}}\rrrr).$$
		\qed
\end{proof}
In the Lemma 3 we use the expansion formulas of $W_n^0$ and $W_n^1$ to obtain the expansions of order $2+\aaaa$ of the weights $\gggg_{n-1}^{(\aaaa)}$ and $\gggg_{n}^{(\aaaa)}$ of approximation \eqref{L8}.
	\begin{lem}
	\be \label{L41}
	\gggg_{n-1}^{(\aaaa)}=\dddd{26-\aaaa}{24\GGGG(-\aaaa)n^{1+\aaaa}}+O\llll(\dddd{1}{n^{2+\aaaa}}\rrrr),
	\ee
	\be\label{L42}
	\gggg_{n}^{(\aaaa)}=-\dddd{1}{\GGGG(1- \aaaa) n^\aaaa}+\dddd{13\aaaa+10}{24\GGGG(-\aaaa)n^{1+\aaaa}}+O\llll(\dddd{1}{n^{2+\aaaa}}\rrrr).
	\ee
	\end{lem}
	\begin{proof}
				$$\gggg_{n}^{(\aaaa)}=W_n^1-W_n^0=-\dddd{1}{\GGGG(1- \aaaa) n^\aaaa}+\llll(\dddd{\aaaa+1}{2}+\dddd{\aaaa-2}{24}\rrrr)\dddd{1}{\GGGG(-\aaaa)n^{1+\aaaa}}+O\llll(\dddd{1}{n^{2+\aaaa}}\rrrr),$$
								$$\gggg_{n}^{(\aaaa)}=-\dddd{1}{\GGGG(1- \aaaa) n^\aaaa}+\dddd{13\aaaa+10}{24\GGGG(-\aaaa)n^{1+\aaaa}}+O\llll(\dddd{1}{n^{2+\aaaa}}\rrrr).$$
From \eqref{L10} with $n:=n-1$
	$$w_{n-1}^{(\aaaa)}=(-1)^n\binom{\aaaa}{n-1}=\dddd{1}{\GGGG(-\aaaa)(n-1)^{1+\aaaa}}+O\llll(\dddd{1}{(n-1)^{2+\aaaa}}\rrrr),$$
		$$w_{n-1}^{(\aaaa)}=\dddd{1}{\GGGG(-\aaaa)n^{1+\aaaa}}+O\llll(\dddd{1}{n^{2+\aaaa}}\rrrr).$$
Therefore
			$$\gggg_{n-1}^{(\aaaa)}=w_{n-1}^{(\aaaa)}-W_n^1=\dddd{1}{\GGGG(-\aaaa)n^{1+\aaaa}}-\dddd{\aaaa-2}{24\GGGG(-\aaaa)n^{1+\aaaa}}+O\llll(\dddd{1}{n^{2+\aaaa}}\rrrr),$$
						$$\gggg_{n-1}^{(\aaaa)}=\dddd{26-\aaaa}{24\GGGG(-\aaaa)n^{1+\aaaa}}+O\llll(\dddd{1}{n^{2+\aaaa}}\rrrr).$$
						\qed
\end{proof}
Now we construct a second-order  approximation of the Caputo derivative at the point $x_n-\aaaa h/2$ by replacing the weights $\gggg_k^{(\aaaa)}$ of approximation \eqref{L33} which have an index $k>\left\lceil N/5\right\rceil$ with the first terms of their expansion formulas \eqref{L10}, \eqref{L41} and \eqref{L42}:
			\be\label{L43}
			\dddd{1}{h^\aaaa}\sum_{k=0}^n \bar{\gggg}_k^{(\aaaa)}y_{n-k}=y^{(\aaaa)}_{n-\aaaa/2}+O\llll(h^2\rrrr).
				\ee
	When $n\leq \left\lceil N/5\right\rceil$ the weights of approximation \eqref{L43} are equal to the weights  of approximation \eqref{L33}: $
	\bar{\gggg}_k^{(\aaaa)}=
w_k^{(\aaaa)}=(-1)^k\binom{\aaaa}{k}$ for $ 0\leq k\leq n-2$ and
$$\bar{\gggg}_{n-1}^{(\aaaa)}=\displaystyle{\dddd{n-2\aaaa}{1+\aaaa-n}w_{n-2}^{(\aaaa-2)}+\dddd{(n-\aaaa/2)^{1-\aaaa}}{\GGGG(2-\aaaa)}},\bar{\gggg}_n^{(\aaaa)}=w_{n-2}^{(\aaaa-2)}-\dddd{(n-\aaaa/2)^{1-\aaaa}}{\GGGG(2-\aaaa)}.$$
	When $n> \left\lceil N/5\right\rceil$ the weights $\bar{\gggg}_k^{(\aaaa)}$  of approximation \eqref{L43} satisfy
	\begin{equation*}\label{L1_MM}
	\bar{\gggg}_k^{(\aaaa)}=\left\{
	\begin{array}{l l}
\displaystyle{w_k^{(\aaaa)}=(-1)^k\binom{\aaaa}{k}},& 0\leq k\leq \left\lceil N/5\right\rceil,\\
\dddd{1}{\GGGG(-\aaaa)k^{1+\aaaa}}-\dddd{\aaaa}{2\GGGG(-1-\aaaa)k^{2+\aaaa}},&  \left\lceil N/5\right\rceil<k=n-2,\\
	\end{array}
		\right .
	\end{equation*}
	$$\bar{\gggg}_{n-1}^{(\aaaa)}=\dddd{26-\aaaa}{24\GGGG(-\aaaa)n^{1+\aaaa}},  \bar{\gggg}_n^{(\aaaa)}=
-\dddd{1}{24\GGGG(1-\aaaa)n^{\aaaa}}+\dddd{13\aaaa+10}{24\GGGG(-\aaaa)n^{1+\aaaa}}.$$
\section{Numerical results}
	The fractional relaxation equation is a two-term ordinary fractional differential equation, where $0<\aaaa<1$:
\begin{equation}\label{FRE}
	y^{(\aaaa)}(x)+L y(x)=f(x),\quad y(0)=y_0.
	\end{equation}
	From \eqref{L2} with $n=1$ we obtain an approximation of order $2-\aaaa$ for the Caputo derivative  at the point at $x=h$:
\begin{equation}\label{L3}
y^{(\aaaa)}(h)=\dddd{y(h)-y(0)}{\GGGG(2-\aaaa)h^\aaaa}+O\llll( h^{2-\aaaa}\rrrr).
\end{equation}
By approximating the Caputo derivative in equation \eqref{FRE} with \eqref{L3} we obtain
	$$\dddd{y(h)-y(0)}{\GGGG(2-\aaaa)h^\aaaa}+L y(h)=f(h)+O\llll( h^{2-\aaaa}\rrrr),
	$$
	$$\widetilde{y}_1=\dddd{y(0)+\GGGG(2-\aaaa)h^\aaaa f(h)}{1+\GGGG(2-\aaaa)L h^\aaaa}=y(h)+O\llll( h^2\rrrr).$$
The number $\widetilde{y}_1$ is a second order approximation for the value of the solution of  equation \eqref{FRE} at $x=h$. Suppose that
	\begin{equation}\tag{*}
	\dddd{1}{h^\aaaa}\sum_{k=0}^n \llllll_k^{(\aaaa)}y_{n-k}=y_n^{(\aaaa)}+O\llll( h^{\bbbb}\rrrr)
	\end{equation}
is an approximation of the Caputo derivative of order $\bbbb$, where $\bbbb\leq 2$. In (Dimitrov 2016) we obtain the numerical solution of  equation \eqref{FRE} on the interval $[0,1]$ which uses approximation (*) for the Caputo derivative.
	\begin{equation}\tag{NS1(*)}
u_n=\dddd{1}{\llllll_0^{(\aaaa)}+L h^\aaaa}
\llll(h^\aaaa f_n-\sum_{k=1}^{n-1} \llllll_k^{(\aaaa)}u_{n-k}   \rrrr),\; u_0=y_0, u_1=\widetilde{y}_1
	\end{equation}
 Suppose that
		\begin{equation}\tag{**}
	\dddd{1}{h^\aaaa}\sum_{k=0}^n \llllll_k^{(\aaaa)}y_{n-k}=y^{(\aaaa)}_{n-\aaaa/2}+O\llll( h^{2}\rrrr)
	\end{equation}
is a second-order shifted approximation of the Caputo derivative with a shift parameter $-\aaaa h/2$.
In (Dimitrov 2014) we show that the numerical solution of equation \eqref{FRE} which uses approximation (**) of the Caputo derivative is computed with $u_0=y_0,u_1=\widetilde{y}_1$ and
	\begin{equation}\tag{NS2(**)}
 u_n=\dddd{1}{\llllll_0^{(\aaaa)}+L\llll(1-\frac{\aaaa}{2} \rrrr)h^\aaaa}\llll(h^\aaaa f_n-\dddd{\aaaa L h^\aaaa}{2}u_{n-1}-\sum_{k=1}^{n-1} \llllll_k^{(\aaaa)}u_{n-k} \rrrr) .
	\end{equation}
	{\bf Example 1:}	
\be \label{Example1}
	y^{(\aaaa)}(x)+y(x)=2x^{2+\aaaa}+\GGGG(3+\aaaa)x^2,\; y(0)=0,
	\ee
		Equation \eqref{Example1} has the solution $y(x)=2x^{2+\aaaa}$. The solution of equation \eqref{Example1} satisfies the condition $y(0)=y'(0)=0$.
	The numerical results for the error and the order of second-order numerical solution NS2(9) of equation \eqref{Example1} are presented in Table 1. The numerical results for the error and the order of numerical solution NS2(16) of order $2-\aaaa$ and the second-order numerical solution NS2(17) of equation \eqref{Example1} are  presented  in Table 2 and Table 3.
\begin{table}[ht]
	\caption{Maximum error and order of numerical solution NS2(9) of  equation \eqref{Example1} when $\aaaa=0.2,\aaaa=0.5$ and $\aaaa=0.9$.}
	\small
	\centering
  \begin{tabular}{| l | c  c | c  c | c  c| }
		\hline
		\multirow{2}*{ $\quad \boldsymbol{h}$}  & \multicolumn{2}{c|}{$\boldsymbol{\aaaa=0.2}$} & \multicolumn{2}{c|}{$\boldsymbol{\aaaa=0.5}$}  & \multicolumn{2}{c|}{$\boldsymbol{\aaaa=0.9}$} \\
		\cline{2-7}
   & $Error$ & $Order$  & $Error$ & $Order$  & $Error$ & $Order$ \\
		\hline
$0.00625$    & $6.9\times 10^{-6}$  & $1.9868$  & $0.00002911$         & $1.9742$    & $0.00004546$         & $1.9774$        \\
$0.003125$   & $1.7\times 10^{-6}$  & $1.9934$  & $7.3\times 10^{-6}$  & $1.9871$    & $0.00001145$         & $1.9889$        \\
$0.0015625$  & $4.4\times 10^{-7}$  & $1.9967$  & $1.8\times 10^{-6}$  & $1.9936$    & $2.8\times 10^{-6}$  & $1.9945$        \\
$0.00078125$ & $1.1\times 10^{-7}$  & $1.9983$  & $4.6\times 10^{-7}$  & $1.9968$    & $7.2\times 10^{-7}$  & $1.9973$        \\
\hline
  \end{tabular}
	\end{table}
\begin{table}[ht]
	\caption{Maximum error and order of numerical solution NS2(16) of  equation \eqref{Example1} when $\aaaa=0.2,\aaaa=0.5$ and $\aaaa=0.9$.}
	\small
	\centering
  \begin{tabular}{| l | c  c | c  c | c  c| }
		\hline
		\multirow{2}*{ $\quad \boldsymbol{h}$}  & \multicolumn{2}{c|}{$\boldsymbol{\aaaa=0.2}$} & \multicolumn{2}{c|}{$\boldsymbol{\aaaa=0.5}$}  & \multicolumn{2}{c|}{$\boldsymbol{\aaaa=0.9}$} \\
		\cline{2-7}
   & $Error$ & $Order$  & $Error$ & $Order$  & $Error$ & $Order$ \\
		\hline
$0.00625$    & $7.9\times 10^{-6}$  & $1.9092$  & $0.00007985$         & $1.5675$    & $0.00044332$  & $1.1490$        \\
$0.003125$   & $2.1\times 10^{-6}$  & $1.9031$  & $0.00002728$         & $1.5494$    & $0.00020323$  & $1.1252$        \\
$0.0015625$  & $5.7\times 10^{-7}$  & $1.8965$  & $9.4\times 10^{-6}$  & $1.5358$    & $0.00009398$  & $1.1127$        \\
$0.00078125$ & $1.5\times 10^{-7}$  & $1.8897$  & $3.3\times 10^{-6}$  & $1.5257$    & $0.00004365$  & $1.1063$        \\
\hline
  \end{tabular}
	\end{table}
\begin{table}[ht]
	\caption{Maximum error and order of numerical solution NS2(17) of  equation \eqref{Example1} when $\aaaa=0.2,\aaaa=0.5$ and $\aaaa=0.9$.}
	\small
	\centering
  \begin{tabular}{| l | c  c | c  c | c  c| }
		\hline
		\multirow{2}*{ $\quad \boldsymbol{h}$}  & \multicolumn{2}{c|}{$\boldsymbol{\aaaa=0.2}$} & \multicolumn{2}{c|}{$\boldsymbol{\aaaa=0.5}$}  & \multicolumn{2}{c|}{$\boldsymbol{\aaaa=0.9}$} \\
		\cline{2-7}
   & $Error$ & $Order$  & $Error$ & $Order$  & $Error$ & $Order$ \\
		\hline
$0.00625$    & $4.3\times 10^{-6}$  & $1.9993$  & $9.3\times 10^{-6}$  & $2.0239$    & $0.00001743$         & $2.1067$        \\
$0.003125$   & $1.1\times 10^{-6}$  & $1.9998$  & $2.3\times 10^{-6}$  & $2.0188$    & $4.0\times 10^{-6}$  & $2.1181$        \\
$0.0015625$  & $2.7\times 10^{-7}$  & $2.0000$  & $5.7\times 10^{-7}$  & $2.0143$    & $9.2\times 10^{-7}$  & $2.1256$        \\
$0.00078125$ & $6.7\times 10^{-8}$  & $2.0000$  & $1.4\times 10^{-7}$  & $2.0106$    & $2.1\times 10^{-7}$  & $2.1313$        \\
\hline
  \end{tabular}
	\end{table}
		The exponential function has a Caputo derivative $D^\aaaa e^x=x^{1-\aaaa}E_{1,2-\aaaa}(x)$, where $E_{a,b}(x)$ is the  Mittag-Leffler function
$E_{a,b}=\sum_{n=0}^\infty\dddd{x^n}{\GGGG(a n+b)}$.
	The sine and cosine functions have Caputo derivatives
$$ D^\aaaa \sin x=x^{1-\aaaa}E_{2,2-\aaaa}\llll(-x^2\rrrr),\; D^\aaaa \cos x=-x^{2-\aaaa}E_{2,3-\aaaa}\llll(-x^2\rrrr).
$$
	The digamma function is the logarithmic derivative of the gamma function
	$$\psi (x)=\dddd{d}{dx}\ln \GGGG(x)=\dddd{\GGGG'(x)}{\GGGG(x)}.$$
	The digamma function satisfies:
	$$\psi(1)=-\gggg,\quad\psi(n)=H_{n-1}-\gggg,\quad \psi (x+1)=\psi (x)+\dddd{1}{x},
	$$
	where $H_n=\sum_{k=1}^n \frac{1}{k}$ is the $n$-th harmonic number and $\gggg=0.5772\dots$ is the Euler-Mascheroni constant. The Caputo derivative of order $\aaaa$ of the function $y(x)=x^3 \ln x$ satisfies
$$	D^\aaaa x^3 \ln x=\dddd{\GGGG(4)x^{3-\aaaa}}{\GGGG(4-\aaaa)}(\ln x+\psi(4)-\psi(4-\aaaa)),
$$
$$	D^\aaaa x^3 \ln x=\dddd{x^{3-\aaaa}}{\GGGG(4-\aaaa)}(11+6\ln x-6\gggg-6\psi(4-\aaaa)).
$$
{\bf Example 2:}	
	\begin{align}\label{Example2} y^{(\aaaa)}&(x)+y(x)=\sin x+\cos x+x^3 \ln x+x^{1-\aaaa}E_{2,2-\aaaa}\llll(-x^2\rrrr)- \\
	&x^{2-\aaaa}E_{2,3-\aaaa}\llll(-x^2\rrrr)+\dddd{x^{3-\aaaa}}{\GGGG(4-\aaaa)}(11+6\ln x-6\gggg-6\psi(4-\aaaa)),\; y(0)=1.\nonumber
	\end{align}
	Equation \eqref{Example3} has an exact solution $y(x)=\sin x+\cos x+x^3 \ln x$. The solution satisfies  $y(0)=y'(0)=1$.   The numerical results for the error and the order of  numerical  solution NS1(22) of order $2-\aaaa$ and second-order numerical solutions NS2(25), NS2(30), NS2(37) of equation \eqref{Example3} are  presented  in  Table 4, Table 5, Table 6 and Table 7.
			\begin{table}[ht]
	\caption{Maximum error and order of numerical solution NS1(22) of  equation \eqref{Example2} when $\aaaa=0.2,\aaaa=0.5$ and $\aaaa=0.9$.}
	\small
	\centering
  \begin{tabular}{ |l | c  c | c  c | c  c| }
		\hline
		\multirow{2}*{ $\quad \boldsymbol{h}$}  & \multicolumn{2}{c|}{$\boldsymbol{\aaaa=0.2}$} & \multicolumn{2}{c|}{$\boldsymbol{\aaaa=0.5}$}  & \multicolumn{2}{c|}{$\boldsymbol{\aaaa=0.9}$} \\
		\cline{2-7}
   & $Error$ & $Order$  & $Error$ & $Order$  & $Error$ & $Order$ \\
		\hline
$0.00625$    & $0.00001516$         & $1.8429$  & $0.00009894$         & $1.5343$    & $0.00115363$   & $1.0966$        \\
$0.003125$   & $4.2\times 10^{-6}$  & $1.8429$  & $0.00003438$         & $1.5251$    & $0.00053881$   & $1.0983$        \\
$0.0015625$  & $1.2\times 10^{-6}$  & $1.8395$  & $0.00001201$         & $1.5174$    & $0.00025150$   & $1.0992$        \\
$0.00078125$ & $3.3\times 10^{-7}$  & $1.8332$  & $4.4\times 10^{-6}$  & $1.5029$    & $0.00011736$   & $1.0996$        \\
\hline
  \end{tabular}
	\end{table}	
	\vspace{-0.3cm}
	\begin{table}[ht]
	\caption{Maximum error and order of numerical solution NS1(25) of  equation \eqref{Example2} when $\aaaa=0.2,\aaaa=0.5$ and $\aaaa=0.9$.}
	\small
	\centering
  \begin{tabular}{| l | c  c | c  c | c  c| }
		\hline
		\multirow{2}*{ $\quad \boldsymbol{h}$}  & \multicolumn{2}{c|}{$\boldsymbol{\aaaa=0.2}$} & \multicolumn{2}{c|}{$\boldsymbol{\aaaa=0.5}$}  & \multicolumn{2}{c|}{$\boldsymbol{\aaaa=0.9}$} \\
		\cline{2-7}
   & $Error$ & $Order$  & $Error$ & $Order$  & $Error$ & $Order$ \\
		\hline
$0.00625$    & $0.00001771$         & $1.9405$   & $0.00002919$         & $1.8905$    & $0.00002368$          & $2.1213$       \\
$0.003125$   & $4.5\times 10^{-6}$  & $1.9686$   & $7.6\times 10^{-6}$  & $1.9346$    & $5.6\times 10^{-6}$   & $2.0908$       \\
$0.0015625$  & $1.1\times 10^{-6}$  & $1.9833$   & $1.9\times 10^{-6}$  & $1.9596$    & $1.3\times 10^{-6}$   & $2.0588$        \\
$0.00078125$ & $2.9\times 10^{-7}$  & $1.9910$   & $5.0\times 10^{-7}$  & $1.9742$    & $3.3\times 10^{-7}$   & $2.0368$        \\
\hline
  \end{tabular}
	\end{table}
		\vspace{-0.3cm}
\begin{table}[ht]
	\caption{Maximum error and order of numerical solution NS2(5) of  equation \eqref{Example2} when $\aaaa=0.2,\aaaa=0.5$ and $\aaaa=0.9$.}
	\small
	\centering
  \begin{tabular}{ |l | c  c | c  c | c  c| }
		\hline
		\multirow{2}*{ $\quad \boldsymbol{h}$}  & \multicolumn{2}{c|}{$\boldsymbol{\aaaa=0.2}$} & \multicolumn{2}{c|}{$\boldsymbol{\aaaa=0.5}$}  & \multicolumn{2}{c|}{$\boldsymbol{\aaaa=0.9}$} \\
		\cline{2-7}
   & $Error$ & $Order$  & $Error$ & $Order$  & $Error$ & $Order$ \\
		\hline
$0.00625$    & $3.3\times 10^{-6}$  & $1.9951$  & $6.7\times 10^{-6}$  & $2.0589$    & $0.00001766$          & $2.0905$        \\
$0.003125$   & $8.2\times 10^{-7}$  & $1.9976$  & $1.6\times 10^{-6}$  & $2.0339$    & $4.2\times 10^{-6}$   & $2.0591$        \\
$0.0015625$  & $2.1\times 10^{-7}$  & $1.9988$  & $4.1\times 10^{-7}$  & $2.0173$    & $1.0\times 10^{-6}$   & $2.0363$        \\
$0.00078125$ & $5.1\times 10^{-8}$  & $1.9994$  & $1.0\times 10^{-7}$  & $2.0075$    & $2.5\times 10^{-7}$   & $2.0214$        \\
\hline
  \end{tabular}
	\end{table}
		\vspace{-0.3cm}
\begin{table}[!ht]
	\caption{Maximum error and order of numerical solution NS2(37) of  equation \eqref{Example2} when $\aaaa=0.2,\aaaa=0.5$ and $\aaaa=0.9$.}
	\small
	\centering
  \begin{tabular}{ |l | c  c | c  c | c  c |}
		\hline
		\multirow{2}*{ $\quad \boldsymbol{h}$}  & \multicolumn{2}{c|}{$\boldsymbol{\aaaa=0.2}$} & \multicolumn{2}{c|}{$\boldsymbol{\aaaa=0.5}$}  & \multicolumn{2}{c|}{$\boldsymbol{\aaaa=0.9}$} \\
		\cline{2-7}
   & $Error$ & $Order$  & $Error$ & $Order$  & $Error$ & $Order$ \\
		\hline
$0.00625$    & $0.00009038$         & $1.9180$  & $0.00019044$         & $1.9240$    & $0.00013348$         & $1.9553$        \\
$0.003125$   & $0.00002322$         & $1.9603$  & $0.00004891$         & $1.9611$    & $0.00003389$         & $1.9776$        \\
$0.0015625$  & $5.9\times 10^{-6}$  & $1.9787$  & $0.00001239$         & $1.9801$    & $8.5\times 10^{-6}$  & $1.9889$        \\
$0.00078125$ & $1.5\times 10^{-6}$  & $1.9893$  & $3.1\times 10^{-6}$  & $1.9901$    & $2.1\times 10^{-6}$  & $1.9946$        \\
\hline
  \end{tabular}
	\end{table}\\
	
The analytical solutions of the system of ordinary fractional differential equations
			\begin{equation}\label{System}
	\left|
	\begin{array}{l l}
y^{(\aaaa)}(x)+Ay(x)+Bz(x)=f(x),&  \\
z^{(\aaaa)}(x)+Cy(x)+Dz(x)=g(x),&  \\
	\end{array}
		\right .
	\end{equation}
	is in studied in  (Ert\"{u}rk and Momani 2008; Diethlm et al. 2017).
Now we obtain the numerical solution NS3(*) of \eqref{System} which uses approximation (*) of the Caputo derivative. By
approximating the Caputo derivative of $y(x)$ and $z(x)$ at the point $x_n=nh$ in both equations of \eqref{System} we obtain
\begin{equation}\label{L433}
\dddd{1}{h^\aaaa}\sum_{k=0}^n \llllll_k^{(\aaaa)}y_{n-k}+A y_n+B  z_n=  f_n+O\llll( h^{\bbbb}\rrrr),
\end{equation}
\begin{equation}\label{L44}
\dddd{1}{h^\aaaa}\sum_{k=0}^n \llllll_k^{(\aaaa)}z_{n-k}+C  y_n+D z_n= g_n+O\llll( h^{\bbbb}\rrrr).
\end{equation}
Let $\{u_n\}_{n=0}^N$ and $\{v_n\}_{n=0}^N$ be the numerical solutions of \eqref{System} on the net $\{x_n\}_{n=0}^N$, where $u_n$ is an approximation of $y_n$ and $v_4$ is an approximation for the value of $z_n$. From \eqref{L433} and \eqref{L44} the numbers $u_n$ and $v_n$ satisfy the following system of linear equations:
	\begin{equation}\label{S2}
	\left|
	\begin{array}{l l}
\llll(\llllll_0^{(\aaaa)}+Ah^\aaaa\rrrr) u_n+B h^\aaaa v_n=S_n,&  \\
C h^\aaaa u_n+\llll(\llllll_0^{(\aaaa)}+Dh^\aaaa\rrrr) v_n=Q_n,& \\
	\end{array}
		\right .
	\end{equation}
	where
	$$S_n=h^\aaaa f_n-\sum_{k=1}^n \llllll_k^{(\aaaa)}u_{n-k},\quad Q_n= h^\aaaa g_n-\sum_{k=1}^n \llllll_k^{(\aaaa)}v_{n-k}.$$
	Denote $\widetilde{D}=\left(\llllll_0^{(\aaaa)}+A h^\aaaa\right) \left(\llllll_0^{(\aaaa)}+D h^\aaaa\right)-B C h^{2\aaaa}$.
The numbers $u_n$ and $v_n$ are the solutions of \eqref{S2}
\begin{equation}\label{L46}
	u_n=\dddd{1}{\widetilde{D}}\llll(\llll(\llllll_0^{(\aaaa)}+Dh^\aaaa\rrrr) S_n-B h^\aaaa Q_n\rrrr),
	\end{equation}
	\begin{equation}\label{L47}
		v_n=\dddd{1}{\widetilde{D}}\llll(-C h^\aaaa S_n+\llll(\llllll_0^{(\aaaa)}+A h^\aaaa\rrrr) Q_n\rrrr).
		\end{equation}
From \eqref{L46} and \eqref{L47} with $n=1$ and approximation \eqref{L3} for
$y_1^{(\aaaa)}$ and $z_1^{(\aaaa)}$ we obtain
	$$\widetilde{u}_1=\dddd{1}{\widetilde{D}}\llll(\left(\ssss_0^{(\aaaa)}+D h^\aaaa\right) \left(h^\aaaa f(h)+\ssss_0^{(\aaaa)} y_0\right)-B h^\aaaa \left(h^\aaaa g(h)+\ssss_0^{(\aaaa)} z_0\right)\rrrr),$$
	$$\widetilde{v}_1=\dddd{1}{\widetilde{D}}\llll(\left(\ssss_0^{(\aaaa)}+A h^\aaaa\right) \left(h^\aaaa g(h)+\ssss_0^{(\aaaa)} z_0\right)-C h^a \left(h^\aaaa f(h)+\ssss_0^{(\aaaa)} y_0\right)\rrrr),$$
	where $\ssss_0^{(\aaaa)}=1/\GGGG(2-\aaaa)$. The numbers $\widetilde{u}_1$ and $\widetilde{v}_1$ are second-order approximations for the values of the solutions $y_1$ and $z_1$ of \eqref{System}.
	Numerical solution NS3(*) has initial conditions $u_0=y_0,u_1=\widetilde{u}_1,v_0=z_0,v_1=\widetilde{v}_1$ and and the numbers $u_n$ and $v_n$, for $2\leq n\leq N$ are computed with \eqref{L46} and \eqref{L47}.\\

		{\bf Example 3:}
			\begin{equation}\label{Example3}
	\left|
	\begin{array}{l l}
y^{(\aaaa)}(x)+y(x)+2z(x)=2e^x+e^{2x}+2^\aaaa x^{1-\aaaa}E_{1,2-\aaaa}(2x),&  \\
z^{(\aaaa)}(x)+3y(x)+4z(x)=4e^x+3e^{2x}+ x^{1-\aaaa}E_{1,2-\aaaa}(x).&  \\
	\end{array}
		\right .
	\end{equation}
	The solution of the sytem of equations \eqref{Example3} is $y(x)=e^{2x},z(x)=e^{x}$.
The numerical results for the error and the order of  numerical  solution NS3(22) of order $2-\aaaa$ and the second-order numerical solution NS3(25) of the solution $y(x)=e^{2x}$ of the system of fractional differential equations \eqref{Example3} are  presented  in Table 8 and Table 9.

	Now we obtain the numerical solution NS4(**) of
	\eqref{System} which uses the second-order shifted approximation (**) of the Caputo derivative. By approximating $y^{(\aaaa)}_{n-\aaaa/2}$ and $z^{(\aaaa)}_{n-\aaaa/2}$ with (**)  we obtain
	\begin{equation}\label{L49}
	\dddd{1}{h^\aaaa}\sum_{k=0}^n \llllll_k^{(\aaaa)}y_{n-k}+A y_{n-\aaaa/2}+B z_{n-\aaaa/2}=f_{n-\aaaa/2}+O\llll(h^2\rrrr),
	\end{equation}
\begin{equation}\label{L50}
		\dddd{1}{h^\aaaa}\sum_{k=0}^n \llllll_k^{(\aaaa)}z_{n-k}+C y_{n-\aaaa/2}+D z_{n-\aaaa/2}=g_{n-\aaaa/2}+O\llll(h^2\rrrr).
		\end{equation}
From \eqref{L49} and \eqref{L50} and the second order approximations	
		$$y_{n-\aaaa/2}=\dddd{\aaaa}{2} y_{n-1}+\llll(1-\dddd{\aaaa}{2} \rrrr)y_n+O\llll(h^2\rrrr),z_{n-\aaaa/2}=\dddd{\aaaa}{2} z_{n-1}+\llll(1-\dddd{\aaaa}{2} \rrrr)z_n+O\llll(h^2\rrrr),	$$
	we obtain the system of equations for $u_n$ and $v_n$
		\begin{equation*}
	\left|
	\begin{array}{l l}
\llll( \llllll_0^{(\aaaa)}+A h^\aaaa \llll( 1-\dddd{\aaaa}{2}\rrrr)\rrrr) u_{n}+B h^\aaaa \llll( 1-\dddd{\aaaa}{2}\rrrr) v_{n}=S_n,&  \\
C h^\aaaa \llll( 1-\dddd{\aaaa}{2}\rrrr) u_{n}+\llll( \llllll_0^{(\aaaa)}+D h^\aaaa \llll( 1-\dddd{\aaaa}{2}\rrrr)\rrrr) v_{n}=Q_n,&  \\
	\end{array}
		\right .
	\end{equation*}
	where
	$$S_n=h^\aaaa\llll(f_{n-\aaaa/2}-\dddd{\aaaa}{2}(A u_{n-1}+B v_{n-1})   \rrrr)-\sum_{k=1}^n \llllll_k^{(\aaaa)}u_{n-k},$$
		$$Q_n=h^\aaaa\llll(g_{n-\aaaa/2}-\dddd{\aaaa}{2}(C u_{n-1}+D v_{n-1})   \rrrr)-\sum_{k=1}^n \llllll_k^{(\aaaa)}v_{n-k}.$$
Denote
$$\widehat{D}=\llll( \llllll_0^{(\aaaa)}+A h^\aaaa \llll( 1-\dddd{\aaaa}{2}\rrrr)\rrrr)\llll( \llllll_0^{(\aaaa)}+D h^\aaaa \llll( 1-\dddd{\aaaa}{2}\rrrr)\rrrr)-BC\llll( 1-\dddd{\aaaa}{2}\rrrr)^2h^{2\aaaa}.$$
Numerical solution NS4(**) of the system of fractional differential equations \eqref{System} has initial conditions
	$u_0=y_0,u_1=\widetilde{u}_1,v_0=z_0,v_1=\widetilde{v}_1$. The numbers $u_n$ and $v_n$ are computed with:
	$$u_n=\dddd{1}{\widehat{D}}\llll(\llll( \llllll_0^{(\aaaa)}+D h^\aaaa \llll( 1-\dddd{\aaaa}{2}\rrrr)\rrrr)S_n-Bh^\aaaa \llll( 1-\dddd{\aaaa}{2}\rrrr)Q_n\rrrr),$$
		$$v_n=\dddd{1}{\widehat{D}}\llll(-C h^\aaaa \llll( 1-\dddd{\aaaa}{2}\rrrr)S_n+\llll( \llllll_0^{(\aaaa)}+A h^\aaaa \llll( 1-\dddd{\aaaa}{2}\rrrr)\rrrr)Q_n\rrrr).$$
The numerical results for the error and the order of second-order numerical solutions NS4(30) and  NS4(37) of the solution $y(x)$ of the system of fractional differential equations \eqref{Example3} are  presented  in Table 10 and Table 11.
	\begin{table}[!ht]
	\caption{Maximum error and order of numerical solution NS3(22) of the solution $y(x)$ of \eqref{Example3} when $\aaaa=0.2,\aaaa=0.5$ and $\aaaa=0.9$.}
	\small
	\centering
  \begin{tabular}{ |l | c  c | c  c | c  c| }
		\hline
		\multirow{2}*{ $\quad \boldsymbol{h}$}  & \multicolumn{2}{c|}{$\boldsymbol{\aaaa=0.2}$} & \multicolumn{2}{c|}{$\boldsymbol{\aaaa=0.5}$}  & \multicolumn{2}{c|}{$\boldsymbol{\aaaa=0.9}$} \\
		\cline{2-7}
   & $Error$ & $Order$  & $Error$ & $Order$  & $Error$ & $Order$ \\
		\hline
$0.00625$    & $0.00023740$  & $1.6637$  & $0.00213254$  & $1.4523$    & $0.02108710$   & $1.0981$        \\
$0.003125$   & $0.00007385$  & $1.6846$  & $0.00077199$  & $1.4659$    & $0.00984709$   & $1.0986$        \\
$0.0015625$  & $0.00002267$  & $1.7038$  & $0.00027752$  & $1.4761$    & $0.00459682$   & $1.0991$        \\
$0.00078125$ & $6.9\times 10^{-6}$  & $1.7201$  & $0.00009927$  & $1.4831$    & $0.00214536$   & $1.0994$        \\
\hline
  \end{tabular}
	\end{table}
	\begin{table}[!ht]
	\caption{Maximum error and order of numerical solution NS3(25) of the solution $y(x)$ of \eqref{Example3} when $\aaaa=0.2,\aaaa=0.5$ and $\aaaa=0.9$.}
	\small
	\centering
  \begin{tabular}{| l | c  c | c  c | c  c| }
		\hline
		\multirow{2}*{ $\quad \boldsymbol{h}$}  & \multicolumn{2}{c|}{$\boldsymbol{\aaaa=0.2}$} & \multicolumn{2}{c|}{$\boldsymbol{\aaaa=0.5}$}  & \multicolumn{2}{c|}{$\boldsymbol{\aaaa=0.9}$} \\
		\cline{2-7}
   & $Error$ & $Order$  & $Error$ & $Order$  & $Error$ & $Order$ \\
		\hline
$0.00625$    & $0.00003134$         & $1.9964$   & $0.00005672$         & $2.0110$    & $0.00008675$          & $2.0370$       \\
$0.003125$   & $7.8\times 10^{-6}$  & $1.9983$   & $0.00001409$         & $2.0086$    & $0.00002119$   & $2.0355$       \\
$0.0015625$  & $1.9\times 10^{-6}$  & $1.9993$   & $3.5\times 10^{-6}$  & $2.0065$    & $5.2\times 10^{-6}$   & $2.0311$        \\
$0.00078125$ & $4.9\times 10^{-6}$  & $1.9997$   & $8.7\times 10^{-7}$  & $2.0048$    & $1.2\times 10^{-6}$   & $2.0292$        \\
\hline
  \end{tabular}
	\end{table}
	\begin{table}[ht]
	\caption{Maximum error and order of numerical solution NS4(5) of the solution $y(x)$ of \eqref{Example3} when $\aaaa=0.2,\aaaa=0.5$ and $\aaaa=0.9$.}
	\small
	\centering
  \begin{tabular}{ |l | c  c | c  c | c  c| }
		\hline
		\multirow{2}*{ $\quad \boldsymbol{h}$}  & \multicolumn{2}{c|}{$\boldsymbol{\aaaa=0.2}$} & \multicolumn{2}{c|}{$\boldsymbol{\aaaa=0.5}$}  & \multicolumn{2}{c|}{$\boldsymbol{\aaaa=0.9}$} \\
		\cline{2-7}
   & $Error$ & $Order$  & $Error$ & $Order$  & $Error$ & $Order$ \\
		\hline
$0.00625$    & $2.7\times 10^{-6}$  & $1.9959$  & $6.1\times 10^{-6}$  & $1.9668$    & $0.00001588$         & $1.9937$        \\
$0.003125$   & $6.9\times 10^{-7}$  & $1.9979$  & $1.5\times 10^{-6}$  & $1.9747$    & $3.9\times 10^{-6}$  & $1.9963$        \\
$0.0015625$  & $1.7\times 10^{-7}$  & $1.9989$  & $3.9\times 10^{-7}$  & $1.9812$    & $9.9\times 10^{-7}$  & $1.9979$        \\
$0.00078125$ & $4.3\times 10^{-8}$  & $1.9994$  & $9.9\times 10^{-8}$  & $1.9863$    & $2.5\times 10^{-7}$  & $1.9988$        \\
\hline
  \end{tabular}
	\end{table}
	\begin{table}[!ht]
	\caption{Maximum error and order of numerical solution NS4(37) of the solution $y(x)$ of \eqref{Example3} when $\aaaa=0.2,\aaaa=0.5$ and $\aaaa=0.9$.}
	\small
	\centering
  \begin{tabular}{ |l | c  c | c  c | c  c |}
		\hline
		\multirow{2}*{ $\quad \boldsymbol{h}$}  & \multicolumn{2}{c|}{$\boldsymbol{\aaaa=0.2}$} & \multicolumn{2}{c|}{$\boldsymbol{\aaaa=0.5}$}  & \multicolumn{2}{c|}{$\boldsymbol{\aaaa=0.9}$} \\
		\cline{2-7}
   & $Error$ & $Order$  & $Error$ & $Order$  & $Error$ & $Order$ \\
		\hline
$0.00625$    & $0.00009952$         & $2.0137$  & $0.00020851$         & $2.0089$    & $0.00015518$         & $1.9852$        \\
$0.003125$   & $0.00002479$         & $2.0055$  & $0.00005198$         & $2.0042$    & $0.00003900$         & $1.9928$        \\
$0.0015625$  & $6.2\times 10^{-6}$  & $2.0018$  & $0.00001298$         & $2.0019$    & $9.8\times 10^{-6}$  & $1.9960$        \\
$0.00078125$ & $1.5\times 10^{-6}$  & $2.0009$  & $3.2\times 10^{-6}$  & $2.0009$    & $2.4\times 10^{-6}$  & $1.9979$        \\
\hline
  \end{tabular}
	\end{table}
	
	 \section{Conclusions}
 In the present paper we showed that the properties of the approximations of the Caputo derivative  are preserved when the weights with an index greater than $\left\lceil N/p\right\rceil$ are replaced by the first  terms of their asymptotic expansions,  where $p$ is a positive number.
 In section 4 we obtained an approximation \eqref{L8} of the Caputo derivative by modifying  the last two weights of the  Gr\"{u}nwald-Letnikov approximation. Approximation \eqref{L8} is a second-order shifted approximation for the Caputo derivative  for all functions $y\in C^2[0,x]$. In future work we are going to apply  the methods and the approximations of the Caputo derivative  discussed in the paper for numerical solution of fractional differential equations with  singular and non-singular  solutions.

\end{document}